\documentclass[11pt,reqno]{amsart}

\usepackage[T1]{fontenc}
\usepackage[utf8]{inputenc}
\usepackage{amssymb}
\usepackage[shortlabels]{enumitem}
\usepackage{graphicx}
\usepackage{booktabs}
\usepackage{array}
\usepackage{float}
\usepackage{xcolor}
\usepackage[british]{babel}
\usepackage[margin=1.15in]{geometry}
\usepackage[hidelinks]{hyperref}
\usepackage{url}
\usepackage{comment}

\newenvironment{ack}
  {\par\medskip\noindent\textbf{Acknowledgements.}\enskip\ignorespaces}{\par}
\newenvironment{funding}
  {\par\medskip\noindent\textbf{Funding.}\enskip\ignorespaces}{\par}

\newcommand{\vecs}{\boldsymbol{s}}

\newcommand{\vecv}{\boldsymbol{v}}

\newcommand{\vecmu}{\boldsymbol{\mu}}

\newcommand{\veco}{\boldsymbol{0}}
\newcommand{\vecone}{\boldsymbol{1}}

\newcommand{\vecg}{\boldsymbol{g}}

\newcommand{\Mo}{\mathcal{M}}
\newcommand{\statespace}{\Omega}

\newcommand{\IFam}{\mathfrak{I}}
\newcommand{\mci}{\mathfrak{I}^{\mathrm{Par-Is}}}
\newcommand{\mcib}{\mathfrak{I}^{\mathrm{B-Par-Is}}}
\newcommand{\mcp}{\mathfrak{I}^{q-\mathrm{Potts}}}
\newcommand{\mcpt}{\mathfrak{I}^{2-\mathrm{Potts}}}
\newcommand{\mca}{\mathfrak{I}^{\mathrm{And-Is}}}
\newcommand{\sset}{spin value set }

\newcommand{\PhiIs}{\Phi_{\mathrm{Par-Is}}}
\newcommand{\PhibIs}{\Phi_{\mathrm{B-Par-Is}}}
\newcommand{\PhiPotts}{\Phi_{\mathrm{q-Potts}}}
\newcommand{\PhiAnd}{\Phi_{\mathrm{And-Is}}}

\theoremstyle{plain}
\newtheorem{theorem}{Theorem}[section]
\newtheorem{proposition}[theorem]{Proposition}

\newtheorem{corollary}[theorem]{Corollary}

\theoremstyle{definition}
\newtheorem{definition}[theorem]{Definition}
\newtheorem{example}[theorem]{Example}
\newtheorem{remark}[theorem]{Remark}

\numberwithin{equation}{section}

\begin{document}

\title[Higher-order spin models via hypergraphs and the Tutte polynomial]
      {A formal framework for higher-order spin models via hypergraphs, polymatroids, and the Tutte polynomial}

\author[K. Berrekkal]{Khallil Berrekkal}
\address{Korteweg-de Vries Institute for Mathematics, University of Amsterdam,
  Science Park 105-107, 1098 XG Amsterdam, The Netherlands}
\email{k.berrekkal@uva.nl}

\author[J. A. Ellis-Monaghan]{Joanna A. Ellis-Monaghan}
\address{Korteweg-de Vries Institute for Mathematics, University of Amsterdam,
  Science Park 105-107, 1098 XG Amsterdam, The Netherlands}
\email{j.a.ellismonaghan@uva.nl}

\author[M. Moody]{Merijn Moody}
\address{Korteweg-de Vries Institute for Mathematics, University of Amsterdam,
  Science Park 105-107, 1098 XG Amsterdam, The Netherlands;
  Institute of Physics, University of Amsterdam, Science Park 904,
  1098 XH Amsterdam, The Netherlands;
  Dutch Institute for Emergent Phenomena, University of Amsterdam,
  1090 GL Amsterdam, The Netherlands}
\email{m.moody@uva.nl}

\author[C. de Mulatier]{Cl\'elia de Mulatier}
\address{Institute for Theoretical Physics, University of Amsterdam,
  Science Park 904, 1098 XH Amsterdam, The Netherlands;
  Informatics Institute, University of Amsterdam, LAB42,
  1098 XH Amsterdam, The Netherlands;
  Dutch Institute for Emergent Phenomena, University of Amsterdam,
  1090 GL Amsterdam, The Netherlands}
\email{c.m.c.demulatier@uva.nl}



\begin{abstract}
We develop a rigorous mathematical framework for statistical mechanics models on hypergraphs, and give conditions for lifting the classical connection between Potts model partition functions and the Tutte polynomial from graphs to hypergraphs. We define hypergraphical models and their associated partition functions, and extend these to special classes of models induced by families of interaction functions. For boolean interaction families, whose interaction functions map to $\{0,1\}$, we show that the partition function is determined by a combinatorial rank function, and we establish sufficient conditions for a hypergraph deletion-contraction recurrence and for when the rank function defines a polymatroid. We illustrate the theory by applying it to three hypergraph interaction families: Parity Ising, Delta Potts, and And Ising. The induced hypergraphical models are not isomorphic to each other, but the first two reduce to the same graphical Ising models. We identify three polymatroids naturally associated with hypergraphs for these models: respectively, the binary matroid of the incidence matrix over $\mathbb{F}_2$, the hypergraphical polymatroid, and the boolean polymatroid. For graphs, the first two reduce to the classical graphical matroid, and their partition functions recover the multivariate Tutte polynomial. The Tutte polynomial thus admits at least two distinct generalizations for hypergraphs, both satisfying a deletion-contraction recurrence: the Tutte polynomial of the binary matroid of the hypergraph incidence matrix, and a multivariate version of the Poincar\'e polynomial of the hypergraphical polymatroid. The partition functions of And Ising models are likewise multivariate versions of the Poincar\'e polynomial, here of the boolean polymatroid. These examples illustrate the much greater range of hypergraphical models and underscore the need for the unifying theory.
\end{abstract}

\maketitle

\section{Introduction}
\label{sec:intro}
Motivated by the increasing applications of higher-order models in statistical physics, we provide a rigorous mathematical formalism for spin models on hypergraphs.
For graphs, the connection between the Tutte polynomial and the partition function of the Ising and Potts models is well established, and has been historically highly fruitful (see~\cite{FK72} for the original papers and e.g.~\cite{Bax82, beaudin2010little, Big93, Bol98, zbMATH05956368, Loe10, Tut84, Wel93, WM00, Sok05} for more recent overviews). New research has begun to extend this fertile connection to higher-order systems via hypergraphs and polymatroids (see \cite{kalman2013version,  bernardi2022universal,  guan2023deletion, Berrekkal2026, chavez2021new}). However, to our knowledge no theory of polymatroids and deletion-contraction yet exists 
that spans both general classes of spin models and models on hypergraphs. We build such a theory for the class of spin models whose local energy terms take binary values (such as higher-order Ising and Potts models). For these we show that the partition function is the generating function of a combinatorial invariant of the hypergraph, which we call its \emph{rank function}. We then give conditions under which this generating function satisfies a deletion--contraction recurrence and the rank function is a polymatroid. Our results are a new set of mathematical tools available
for physics applications and a precisely formulated foundation to support further mathematical development.

\subsection{Background}
\label{sec:background}
Statistical mechanics models such as the Ising model~\cite{LenzBeitragZV, isingBeitragZurTheorie1925} and Potts model~\cite{wu1982potts} have been cornerstones of physics for nearly a century. Over time, these models have become increasingly sophisticated, accommodating external fields, edge-dependent interaction energies, and generally increasingly complex interactions. These models were introduced in condensed matter~\cite{Bax82}, where they are mainly studied on lattices, but are now used in many different fields, in particular in the context of statistical modeling of complex systems~\cite{Nguyen2017}, where they are studied more generally on graphs (such as in the reconstruction of protein interactions~\cite{Cocco2018}, the modeling of functional connectivity from brain neuronal activity~\cite{savin2017maximum}, or the reconstruction of psychological networks from psychometric data~\cite{borsboom2021network}).
Recently, there has been a surge of interest in models with high-order interactions, in particular in complex systems (see survey~\cite{battiston2021physics}), where the interactions are not only between pairs of vertices, but among multiple vertices, i.e., models on hypergraphs. On the mathematical side, a parallel line of work has begun to develop the combinatorial counterpart of this shift, extending the Tutte polynomial and its deletion--contraction theory from matroids to polymatroids and hypergraphs~\cite{kalman2013version, bernardi2022universal, guan2023deletion, Berrekkal2026, chavez2021new}.

Let us
briefly recall the classical setup. Consider a system of random variables $s_i$, called \emph{spins}, placed on the vertices of a graph $G = (V,E)$, each taking values in a finite set $S$ with $|S| = q$. The space of configurations of this system is $\statespace = S^{|V|}$. The energy of a configuration $\vecs \in \statespace$ is described by a Hamiltonian $H(\vecs; G, \vecg)$, which depends on a vector of parameters $\vecg$ and on the graph $G$. The Hamiltonian determines a probability distribution over configurations,
\[
    p(\vecs \mid G, \vecg) = \frac{e^{-H(\vecs;\,G, \vecg)}}{Z(G; \vecg)},
    \qquad
    Z(G; \vecg) = \sum_{\vecs \in \statespace} e^{-H(\vecs; G, \vecg)},
\]
where the normalizing factor $Z(G; \vecg)$ is the \emph{partition function}. This is the Boltzmann distribution: it describes the equilibrium statistics of the spin system, with lower-energy configurations more likely to occur. The partition function encodes the statistical properties of the system and is therefore the central object of study.

For example, the classical Ising model~\cite{LenzBeitragZV, isingBeitragZurTheorie1925} has spins taking values in $S = \{-1, +1\}$ and Hamiltonian
\[
H_{\mathrm{Ising}}(\vecs; G, \vecg) = \sum_{\{i,j\} \in E} g_{ij}\, s_i s_j.
\]
The classical $q$-state Potts model~\cite{wu1982potts} has spins taking values in $S = \{1, \dots, q\}$ and Hamiltonian
\[
H_{\mathrm{Potts}}(\vecs; G, \vecg) = \sum_{\{i,j\} \in E} g_{ij}\, \delta(s_i, s_j),
\]
where $\delta(s_i, s_j) = 1$ if $s_i = s_j$ and $0$ otherwise. For both models, the Hamiltonian is defined by a local interaction energy of the form $g_{ij} \phi^{ij}(s_i,s_j)$ for each edge $(i,j)\in E$, where $g_{ij}$ is a parameter and $\phi^{ij}(s_i,s_j)$ is a function that depends only on the spins at the vertices of the edge $(i,j)$.
For $q = 2$, the two models are related by a change of variables. The bijection $\tau \colon \{1,2\} \to \{-1,+1\}$ given by $\tau(s) = 2s - 3$ satisfies $\tau(s_i)\tau(s_j) = 2\delta(s_i, s_j) - 1$, so that
\[
H_{\mathrm{Potts}}(\vecs; G, \vecg) = \frac{1}{2} H_{\mathrm{Ising}}(\hat{\tau}(\vecs); G, \vecg) + \frac{1}{2}\sum_{\{i,j\} \in E} g_{ij},
\]
where $\hat{\tau}$ applies $\tau$ componentwise. The partition functions are then related by
\begin{equation}
\label{eq:introequiv}
Z_{\mathrm{Potts}}(2, G; \vecg) = \exp\!\left(\tfrac{1}{2}\textstyle\sum_{\{i,j\} \in E} g_{ij}\right)\, Z_{\mathrm{Ising}}\!\left(G; \tfrac{\vecg}{2}\right).
\end{equation}

A central object on the combinatorial side is the \emph{multivariate Tutte polynomial} of a graph $G = (V,E)$, defined by
\[
\widetilde{Z}(q, \vecv; G) = \sum_{A \subseteq E} q^{\kappa(A)} \prod_{e \in A} v_e,
\]
where $\kappa(A)$ denotes the number of connected components of the spanning subgraph $(V,A)$ and $\vecv = (v_e)_{e \in E}$ are formal variables~\cite{Sok05}. This polynomial depends on $G$ only through its graphical matroid. The celebrated Fortuin--Kasteleyn representation~\cite{FK72} establishes that the Potts model partition function is an evaluation of the multivariate Tutte polynomial: setting $v_e = e^{g_e} - 1$ one obtains $Z_{\mathrm{Potts}}(q, G; \vecg) = \widetilde{Z}(q, \vecv; G)$. For other works on this connection and more recent overviews, see e.g.~\cite{Bax82, beaudin2010little, Big93, Bol98, Loe10, Tut84, Wel93, WM00, Sok05}.

The equivalence between the Potts model partition function and the multivariate Tutte polynomial catalyzed a fruitful cross-pollination between statistical mechanics and matroid theory \cite{beaudin2010little, WM00, ellis-monaghanHandbookTuttePolynomial2022}. For example, the famous Kramers-Wannier duality can be naturally understood as an instance of matroid duality, while the study of the polynomial's complex zeros allowed researchers to map complete complex-temperature phase diagrams \cite{wegnerDualityGeneralizedIsing1971, CHANG2001234, changExactPottsModel2002, changExactPottsModel2004, salasAbsencePhaseTransition1997, jacobsenComplextemperaturePhaseDiagram2006, biggsEquimodularCurves2002, changPartitionFunctionZeros2006, chenPartitionFunctionZeros1996, martinPottsModelsRelated1991, martinZerosPartitionFunction1986}. Another application area is the computational complexity of the partition functions of Potts models. In general, computing the partition function of a Potts model on a graph $G$ is \#P-Hard, however for certain families of graphs the problem becomes tractable \cite{GJ07, Roy09, Sok05, jerrumTwodimensionalMonomerdimerSystems1987, Jaeger_Vertigan_Welsh_1990, robertsonGraphMinorsExcluding1983, robertsonGraphMinorsIII1984, robertsonGraphMinorsII1986, nobleEvaluatingTuttePolynomial1998}.

Generalizations of the Ising/Potts model to higher-order systems, i.e. hypergraphs $G=(V,E)$, where $E$ is a set of hyperedges, have also been studied in physics.
For these models, the local interaction energy terms defining the Hamiltonian can depend on more than two spins. For example, the generalization of classical $q$-state Potts model~\cite{grimmett1994potts} has spins taking values in $S = \{1, \dots, q\}$ and Hamiltonian
\[
H_{\mathrm{Potts}}(\vecs; G, \vecg) = \sum_{e \in E} g_{e}\, \delta_e(\vecs),
\]
where $\delta_e(\vecs) = 1$ if $s_i = s_j$ for all $i,j\in e$ and $0$ otherwise. The generalization of the Ising model to hypergraphs~\cite{wegnerDualityGeneralizedIsing1971, spinmod} has spins taking values in $S=\{-1,1\}$ and Hamiltonian:
\[
H_{\mathrm{Ising}}(\vecs; G, \vecg) = \sum_{e\in E} g_{e}\, \prod_{i\in e} s_i.
\]
where the product is over the vertices of the hyperedge $e$.
Models of this form have been studied in statistical physics as a \emph{maximum entropy models}~\cite{PhysRev.106.620} for modeling high-dimensional binary data. In this context, they are used to reconstruct higher-order interactions between binary variables from data observations~\cite{mastromatteo2013beyond, Gresele2017, jansma2025high}.
In the most general case, the parameters $g_e$ are different for each hyperedge $e\in E$ and the Hamiltonian also has
interaction terms based on single-vertex hyperedges (blisters), which correspond to interactions with an external field in physics and are notoriously more challenging to treat~\cite{zbMATH05956368}.

Many statistical physics models have a similar form,
where the Hamiltonian is written as a sum over local interaction energies between spin variables so it can be decomposed over a hypergraph (e.g., graphical and hypergraphical generalizations of, the $S=\{0,1\}$-version of the Ising model~\cite{jansma2025mereological}, vector Potts model~\cite{wu1982potts, deClercq2026}, Blume-Capel model~\cite{blume1966theory, capel1966possibility, waldorp2026blume}, Blume-Emery-Griffiths model~\cite{blume1971ising, armanetti2026inverse}, classical XY model~\cite{stanley1968dependence}).
These models fall under the general category of classical (as opposed to quantum) spin models.

\subsection{Overview of results}
\label{sec:overview}

We study systems of spins $s_i$ placed on the vertices of a hypergraph $G = (V,E)$, taking values in a finite \sset $S$, with configuration space $\statespace = S^{|V|}$. For all the models above, the Hamiltonian is a sum of local interaction energies $g_e\, \phi^e(\vecs)$, one per hyperedge $e$, where $\phi^e$ depends only on the spins at the vertices of $e$. The partition function then takes the general form
\[
    Z(G; \vecg) = \sum_{\vecs \in \statespace} \prod_{e \in E} \exp\!\left( g_e\, \phi^e(\vecs) \right).
\]
(We have absorbed the minus sign that conventionally appears in the Boltzmann distribution into $g_e$; this loses no generality.)

In Section~\ref{sec:mcdef} we define
a \emph{hypergraphical model} $\Mo = (G, \{\phi^e\}_{e \in E})$, by pairing a hypergraph $G$ with a set of \emph{interaction functions} $\phi^e \colon \statespace \to \mathbb{C}$, each depending only on the spins at the vertices of its hyperedge $e$.
More generally,
to describe a whole class of models uniformly we introduce \emph{interaction families} $\IFam = (S, \{\Phi^k\}_{k \in \mathbb{Z}_{\geq 0}})$, where $\Phi^k \colon S^k \to \mathbb{C}$ prescribes how any $k$ spins interact. An interaction family induces a model $\Mo(G; \IFam)$ on \emph{every} hypergraph $G$ simply by applying $\Phi^{|e|}$ to each hyperedge.

Our first application of this framework is in Section~\ref{sec:mcdef}, where we illustrate the theory on three interaction families, each generalizing a classical graph model to hypergraphs:
\begin{itemize}
    \item the \emph{Parity Ising} family $\mci$, with $S = \{-1, +1\}$ and $\Phi^k(\vecs) = \prod_{i=1}^k s_i$, generalizing the Ising model~\cite{LenzBeitragZV, isingBeitragZurTheorie1925}; its hypergraph form first appeared in lattice gauge theory~\cite{wegnerDualityGeneralizedIsing1971} and is now used to infer higher-order interactions from binary data~\cite{mastromatteo2013beyond, spinmod, clelbayes};
    \item the \emph{Delta Potts} family $\mcp$, with $S = [q]$ and $\Phi^k(\vecs) = \delta_k(\vecs)$, generalizing the Potts model~\cite{wu1982potts}; hypergraph versions have been studied in~\cite{grimmett1994potts, Berrekkal2026};
    \item the \emph{And Ising} family $\mca$, with $S = \{0,1\}$ and $\Phi^k(\vecs) = \prod_{i=1}^k s_i$, generalizing the lattice-gas form of the Ising model~\cite{lee1952statistical, Bax82}, also widely used for statistical inference in complex systems~\cite{Nguyen2017, jansma2025mereological}.
\end{itemize}
We then define isomorphisms between hypergraphical models and study the relations between these hypergraphical models. The classical $q=2$ equivalence~\eqref{eq:introequiv} between the Ising and the Potts model translates to a hypergraphical model isomorphism $\Mo(G; \mci) \cong \Mo(G; \mcpt)$ (Example~\ref{ex:graphmodequiv}). But this isomorphism is a low-order coincidence: we prove that the \emph{families} $\mci$ and $\mcpt$ are not isomorphic (Proposition~\ref{prop:potts_parity_noniso}). The And Ising family is even further from the other two, with $\Mo(G; \mca) \not\cong \Mo(G; \mci)$ even on graphs.

In the rest of the paper we restrict our focus to interaction families and hypergraphical models whose interaction functions only take two values (such as in the examples above). 
We show that such hypergraphical models are isomorphic to  hypergraphical models induced by \emph{boolean} interaction families, whose interaction functions take value in $\{0, 1\}$ (Remark~\ref{rem:binary_reduction}). We
further restrict ourselves to \emph{loopless} hypergraphs, i.e. hypergraphs in which no hyperedge contains a vertex more than once. For the three interaction families introduced above this is no loss of generality as they are \emph{loopblind} (Definition~\ref{def:loopless}): any interaction on a hyperedge with repeated vertices reduces to an interaction on an ordinary subset of the vertices, so every model is equivalent to one on a loopless hypergraph (Proposition~\ref{prop:loopless_reduction}).
\begin{itemize}
    \item \textbf{Rank function and the partition function.} Our first central result in Section~\ref{sec:delcon} is that the partition function of a boolean hypergraphical model $\Mo$ is governed by a combinatorial function we call the \emph{rank function} $r_\Mo$ on subsets of the hyperedges (Definition~\ref{def:rankfunc}), that counts the number of configurations satisfying a given set of constraints. The partition function is then given by (Proposition~\ref{prop:boolmodcountex}):
    \[
        Z(\Mo; \vecg) = q^{|V|} \sum_{A \subseteq E} q^{-r_\Mo(A)} \prod_{e \in A} v_e,
    \]
    where $q = |S|$ and $v_e = e^{g_e} - 1$. As a direct consequence, two models with the same rank function have the same partition function.

    \item \textbf{Deletion--contraction recurrence.} Our second result in Section~\ref{sec:delcon} is a deletion--contraction recurrence for boolean interaction families.
    Contracting a hyperedge $e$ amounts to fixing the value of its interaction $\phi^e(\vecs)$, which constrains the configuration space of the whole system.
    To rewrite this constrained system as a hypergraphical model of the same interaction family, we introduce the following two conditions:
    \emph{functional representability} allows us to rewrite the constrained configurations as a free subset of the spins, through a map $\sigma$, so that contraction can be realized geometrically by deleting the remaining (dependent) vertices; \emph{contraction-closedness} asks that the interactions of the resulting model again lie in the original interaction family. We obtain the following recurrence
    (Theorem~\ref{theo:contraction_closure}, Corollary~\ref{cor:partition_delcon}): if $\IFam$ is boolean, functionally representable and contraction-closed, and $\Mo = \Mo(G; \IFam)$ for a loopless hypergraph $G$, then for every hyperedge $e$, with deletion $\Mo \setminus e$ and contraction $\Mo /_{\IFam, \sigma}e$,
    \[
        Z(\Mo; \vecg) = Z(\Mo \setminus e; \vecg') + v_e\, Z(\Mo /_{\IFam,\sigma} e; \vecg'),
    \]
    where $v_e = e^{g_e} - 1$ and $\vecg'$ restricts $\vecg$ to $E \setminus \{e\}$.

    \item \textbf{Polymatroid rank function.} In Section~\ref{sec:polrank} we connect the rank functions of boolean hypergraphical models to the theory of polymatroids. In particular, we give sufficient conditions for the rank function $r_\Mo$ to be the rank function of a polymatroid. Normalization and monotonicity hold for any boolean model (Section~\ref{sec:polrank}); submodularity follows from two further conditions on the interaction family, \emph{group--coset} and \emph{global satisfiability}. Finally, integer-valuedness follows from global satisfiability and the deletion--contraction recurrence above so that we have the following result (Corollary~\ref{cor:rank_is_polymatroid}): if $\IFam$ is a loopblind boolean interaction family that is group--coset, globally satisfiable, functionally representable, and contraction-closed, then for every loopless hypergraph $G$ the rank function $r_{\Mo(G;\IFam)}$ is a polymatroid.
\end{itemize}
In Section~\ref{sec:example-interaction-families} we return to the three
families $\mci$, $\mcp$, and $\mca$ and apply our theorems to them. All
three satisfy the hypotheses of our three main results, and our framework recovers a number of results from the literature. The rank functions of the three families are three
distinct classical combinatorial structures commonly associated to hypergraphs (summarized in Table~\ref{tab:models}), giving a physical interpretation of these invariants. 

The boolean equivalent of the Parity Ising interaction family $\mcib$ yields the binary \emph{incidence matroid} over
$\mathbb{F}_2$. Its rank function is in fact a matroid, so its partition
function is an evaluation of the multivariate Tutte polynomial of that
matroid (Remark~\ref{rem:parity_tutte}). Modeling an external field by \emph{blisters} (single-vertex hyperedges) recovers, from a
matroid-theoretic viewpoint, the deletion--contraction results
of~\cite{zbMATH05956368} for the Ising model in an external magnetic
field (Section~\ref{Section:blisters}). The same matroid can be
represented by many different hypergraphs, which therefore share a
partition function. This was observed in~\cite{spinmod}, where the basis
transformations of the incidence matrix that preserve the matroid are
termed \emph{gauge transformations} and shown to leave the partition
function invariant~\cite{spinmod, deClercq2026}.

Delta Potts $\mcp$ yields the \emph{hypergraphical polymatroid},
introduced as the \emph{chromatic hypermatroid} in~\cite{Helg}; its partition functions are multivariate versions of the \emph{Poincar\'e
polynomial}~\cite{Helg} of that polymatroid (Remark~\ref{rem:delta_potts_poincare}), and its hypergraph contraction
agrees with that of the hypergraph Tutte polynomial
of~\cite{Berrekkal2026}. And Ising $\mca$ yields the \emph{boolean
polymatroid}, the \emph{covering hypermatroid} of~\cite{Helg}; its partition functions are again multivariate versions of the
Poincar\'e polynomial of that polymatroid
(Remark~\ref{rem:and_ising_poincare}). We follow~\cite{Vertigan_Whittle_1993}
in calling these polymatroids the hypergraphical and boolean polymatroids. Evaluating
the Poincar\'e polynomial at $v_e = -1$ gives known counting identities:
the number of weak $q$-colorings for $\mcp$ and the number of transversal
sets for $\mca$, matching~\cite{Helg}.

When restricted to graphs, both $\mci$ and $\mcp$ reduce to the classical
graphical matroid, and as they are not isomorphic as interaction families their partition functions give \emph{two genuinely
distinct} hypergraph generalizations of the multivariate Tutte
polynomial.

The paper is organized as follows. Section~\ref{sec:mcdef} sets up hypergraphical models, interaction families, and the boolean and loopblind subclasses, and treats the three running examples. Section~\ref{sec:delcon} develops the rank function and the deletion--contraction recurrence. Section~\ref{sec:polrank} gives the conditions under which the rank function is a polymatroid. Section~\ref{sec:example-interaction-families} applies the theory to the three families and draws out the connections to the literature summarized above. Section~\ref{sec:discussion} concludes with a discussion and further directions of investigation.

\section{Hypergraphical models and interaction families}
\label{sec:mcdef}
In this section we set up the definitions for our theory of interaction families. We start by defining hypergraphical models, which provide a way to associate a family of Hamiltonians to a given hypergraph. We then introduce interaction families which induce hypergraphical models for any given hypergraph. Finally we introduce boolean interaction families for which we shall establish a `Tutte-polynomial-like' theory in Section \ref{sec:delcon} and Section \ref{sec:polrank}.

\subsection{Hypergraph conventions}
We start by fixing our hypergraph notation.
\begin{definition}  \label{simplehypergraph}
A \emph{simple hypergraph}   $H$ consists of a set $V$ whose elements are the \emph{vertices} of the hypergraph, together with a set of \emph{hyperedges} $E$, where each edge is a subset of $V$.
\end{definition} 

\begin{definition} \label{hypergraph}
 A \emph{hypergraph}   $H$ consists of a set $V$ whose elements are called the \emph{vertices} of the hypergraph, together with a multiset of \emph{hyperedges} $E$ that are  multisets of the elements of $V$.  
\end{definition}

\begin{definition}
    A \emph{loop} of a hypergraph $H = (V,E)$ is an edge $e\in E$ that is a multiset but not a subset of the elements of $V$, i.e. at least one element of $V$ appears more than once in $e$.
\end{definition}

\begin{definition}
    A \emph{blister} of a hypergraph $H = (V,E)$ is an edge $e$ consisting of precisely one vertex.
\end{definition}

\begin{definition}
    A \textit{loopless hypergraph} is a hypergraph that contains no loops.
\end{definition}

Note that a simple hypergraph is then just a hypergraph where each edge is a set and all the edges together form a set. If it is important to specify the vertices and edges of a particular hypergraph, we may write $H(V, E)$ or denote them by $V(H)$ and $E(H)$. Thus, except for blisters, our terminology reduces to standard conventions when the hypergraphs are graphs.

A hypergraph $H$ is a \emph{subhypergraph} of $G$ if $V(H) \subseteq V(G)$ and $E(H) \subseteq E(G)$ (as multisets). We denote the set of vertices incident to some edge $e \in E(G)$ by $\overline{e}$ and the set of vertices incident to some subset of edges $A \subset E(G)$ by $\overline{A}$. The subhypergraph \emph{induced} by $A \subseteq E(G)$ has edge set $E(H) = A$ and vertex set $V(H) = \overline{A}$ consisting of exactly the vertices in $G$ that are contained in some edge of $E(H)$. It is denoted by $G[A]$. The \emph{spanning subhypergraph} induced by $A \subseteq E(G)$ has $E(H) = A$ and $V(H) = V(G)$. We denote it by $G_A$. This distinction between the subgraph induced by $A$ and the spanning subgraph induced by $A$  is important.

A \emph{connected component} of a hypergraph $G$ is a minimal non-empty subset $V' \subseteq V(G)$ such that for every $e \in E(G)$, either $\overline{e} \cap V' = \emptyset$ or $\overline{e} \subseteq V'$.
Equivalently, two vertices $v_1, v_2 \in V$ are in the same connected component of $G$ if and only if $v_1 = v_2$ or there exists a sequence of hyperedges $e_1, \dots, e_k$ in $E(G)$ such that $v_1 \in e_1$, $v_2 \in e_k$, and $\overline{e_{i}} \cap \overline{e_{i+1}} \neq \emptyset$ for all $1 \leq i < k$.
We denote the number of connected components of $G$ by $\kappa(G)$. 

Definition \ref{hypergraph} allows for hypergraph analogs of loops (hyperedges that can contain the same vertex multiple times) and duplicate edges (more than one hyperedge comprising the same set of vertices), and also permits empty hyperedges.  In Proposition \ref{prop:edge_redundancy} we will see that duplicate and empty hyperedges may be absorbed into the interactions functions given the locality conditions of Definition \ref{def:graphmod}.  Formulas and results depending on the number of hyperedges may need to account for the possibility of empty hyperedges; to avoid this technicality, it may sometimes be more convenient simply to require that all hyperedges be non-empty. 

Since graphs are a subclass of hypergraphs, we generally will not use distinguishing letters for the classes, so that $H$ may be a graph, while $G$ may be hypergraph.

\subsection{Hypergraphical models}
Here we define hypergraphical models, which give us a formal way to associate a parameterized family of Hamiltonians to a given hypergraph. We consider systems of random variables~$s_i$, called spins, placed on the vertices $i$ of a hypergraph $G = (V,E)$. The spins take value in the \sset $S$. We let $\statespace = S^{|V|}$ denote the state space of the system. The key idea is that the hyperedges encode which groups of spins can interact: each hyperedge carries a local interaction function that depends only on the spins of its incident vertices.

\begin{definition} \label{def:graphmod}
A \emph{hypergraphical model} $\Mo$ is a pair $(G, \{\phi^e\}_{e \in E})$ where $G = (V, E)$ is a hypergraph whose vertices index the spins of a system with states $\vecs \in \statespace$, and $\{\phi^e\}_{e \in E}$ is a family of \emph{interaction functions} where each $\phi^e \colon \statespace \to \mathbb{C}$ satisfies the following locality conditions:
\begin{enumerate}
    \item\label{cond:graphmod:locality} $\phi^e(\vecs) = \phi^e(\vecs')$ if $\{s_i \}_{i \in e} = \{s'_i \}_{i \in e}$ as multisets,
    
    \item\label{cond:graphmod:duplicates} $\phi^e(\vecs) = \phi^{e'}(\vecs)$ for all $\vecs$ if $e = e'$, i.e., if $e$ and $e'$ are duplicate hyperedges.
\end{enumerate}
\end{definition}
\begin{remark}
    \label{rem:symmetry}
    Note that a function $f: S^k \to \mathbb{C}$ depends only on the multiset of values $\{s_1, \dots, s_k\}$ if and only if $f$ is symmetric under any permutation $\sigma \in S_k$ of its arguments.
\end{remark}
Condition~\ref{cond:graphmod:locality} ensures that the interaction functions are local, they only depend on the vertices in their hyperedge, and symmetric under permutation of the vertices in their hyperedge. Condition~\ref{cond:graphmod:duplicates} is a technical condition that ensures parallel hyperedges have the same interaction functions.

The interaction functions of a graphical model $\Mo$ define a parameterized family of Hamiltonians, which in turn define a family of maximum entropy distributions as follows.  

\begin{definition}
\label{def:maxent}
Let $\Mo = (G, \{\phi^e\}_{e \in E})$ be a hypergraphical model. For a parameter vector $\vecg = (g_e)_{e \in E} \in \mathbb{C}^{|E|}$ such that $g_e \phi^e(\vecs) \in \mathbb{R}$ for all $\vecs \in \statespace$ the \emph{maximum-entropy distribution} of $\Mo$ is:
\[
    p(\vecs \mid \Mo, \vecg) = \frac{1}{Z(\Mo; \vecg)} \prod_{e \in E} \exp\left( g_e \phi^e(\vecs) \right),
\]
where the \emph{model partition function} $Z(\Mo ; \vecg)$ is:
\[
    Z(\Mo; \vecg) = \sum_{\vecs \in \Omega} \prod_{e \in E} \exp\left( g_e \phi^e(\vecs) \right).
\]
\end{definition}

The parameters $g_e$ and function $\phi^e(\vecs)$ usually take real values, as $p(\vecs\,|\,\Mo, \vecg)$ is a probability distribution.
However, it is also common to study the partition function as a complex-valued function to study phase transitions in classical many-body physics~\cite{yang1952statistical} and complex systems~\cite{krasnytska2015violation, krasnytska2016partition}, and recent classical spin models were also introduced with complex-valued parameters and interaction functions~\cite{deClercq2026}.
In the following, we relax the constraints that $Z$ must be real. 

\begin{remark}[Relation to Statistical Inference]
    In the context of the statistical modeling of complex systems, the probability distributions in Definition~\ref{def:maxent} are often used to infer the parameters $\vecg$ from a finite dataset of empirical observations, see e.g. \cite{Nguyen2017, spinmod, clelbayes}. In this setting, there is a dataset of observations of the state space $\Omega$ and the goal is to recover the underlying probability distribution. One approach is to tune the parameters $\vecg$ so that the expectation values of the operators $\phi^e(\vecs)$ equal the empirical averages measured from the dataset \cite{PhysRev.106.620}. The partition function $Z(\Mo; \vecg)$ plays a central role here, not just as a normalization constant, but because its derivatives yield the statistical moments of the distribution. For instance, the Hessian of $\log Z(\Mo; \vecg)$ gives the Fisher information matrix, which characterizes the information-theoretic complexity of the model in parameter space.
\end{remark}

We now consider the conditions under which two hypergraphical models, whose variables can take values in different spaces $S$ and $S'$, are isomorphic, essentially when the underlying hypergraphs are isomorphic, there is a compatible bijection between the spins, and the interaction functions are linearly related.

\begin{definition}[Hypergraphical Model Isomorphism]
Let $\Mo = (G = (V,E), \{\phi^e\}_{e \in E})$ and $\Mo' = (G' = (V',E'), \{\psi^{e'}\}_{e' \in E'})$ be two hypergraphical models with state spaces $\statespace = S^{|V|}$ and $\statespace' = S'^{|V'|}$ respectively. We say that $\Mo$ and $\Mo'$ are \emph{isomorphic}, denoted $\Mo \cong \Mo'$, if there exist:
\begin{enumerate}
    \item A hypergraph isomorphism, consisting of a permutation of the vertices $\pi: V \to V'$ inducing a bijection $\hat{\pi}: E \to E' \colon e \mapsto \{ \pi (i) \colon i \in e \}$,
    \item A bijection $\tau: S \to S'$, which induces a bijection $\hat{\tau} \colon \statespace \to \statespace' \colon (s_i)_{i\in V} \mapsto (\tau(s_{\pi^{-1}(i)}))_{i \in V'}$, 
    \item Constants $\alpha_e, \beta_e \in \mathbb{R}$ for all $e \in E$ with $\alpha_e \neq 0$,
\end{enumerate}
such that for all $e \in E$ and all $\vecs \in \statespace$, the interaction functions satisfy the linear relation:
\begin{equation}
    \phi^e(\vecs) = \alpha_e \cdot \psi^{\hat{\pi}(e)}(\hat{\tau}(\vecs)) + \beta_e.
\end{equation}
\end{definition}

If two models are isomorphic, their partition functions are related by a simple scaling of parameters and a multiplicative prefactor as follows.

\begin{proposition}
\label{prop:equivmodpart}
Let $\Mo$ and $\Mo'$ be isomorphic hypergraphical models as defined above, with parameter vectors $\vecg$ and $\vecg'$. If the parameters are related by $g'_{\hat{\pi}(e)} = \alpha_e\, g_e$ for all $e \in E$, then the partition functions satisfy:
\[
    Z(\Mo; \vecg) = \exp\!\left( \sum_{e \in E} \beta_e\, g_e \right) Z(\Mo'; \vecg').
\]
\end{proposition}
\begin{proof}
By definition, $Z(\Mo; \vecg) = \sum_{\vecs \in \statespace} \prod_{e \in E} \exp(g_e \phi^e(\vecs))$. Substituting the isomorphism relation $\phi^e(\vecs) = \alpha_e \psi^{\hat{\pi}(e)}(\hat{\tau}(\vecs)) + \beta_e$:
\begin{align*}
    Z(\Mo; \vecg) &= \sum_{\vecs \in \statespace} \prod_{e \in E} \exp\!\left( g_e \left[ \alpha_e \psi^{\hat{\pi}(e)}(\hat{\tau}(\vecs)) + \beta_e \right] \right) \\
    &= \sum_{\vecs \in \statespace} \prod_{e \in E} \exp\!\left( \alpha_e g_e \, \psi^{\hat{\pi}(e)}(\hat{\tau}(\vecs)) \right) \cdot \exp(\beta_e g_e) \\
    &= \exp\!\left( \sum_{e \in E} \beta_e g_e \right) \sum_{\vecs \in \statespace} \prod_{e \in E} \exp\!\left( \alpha_e g_e \, \psi^{\hat{\pi}(e)}(\hat{\tau}(\vecs)) \right).
\end{align*}
Since $\hat{\tau} \colon \statespace \to \statespace'$ is a bijection, the change of variables $\vecs' = \hat{\tau}(\vecs)$ rewrites the sum as one over $\vecs' \in \statespace'$. Likewise, since $\hat{\pi} \colon E \to E'$ is a bijection, the product over $e \in E$ becomes a product over $e' = \hat{\pi}(e) \in E'$. Substituting $g'_{e'} = \alpha_e g_e$ yields:
\[
    Z(\Mo; \vecg) = \exp\!\left( \sum_{e \in E} \beta_e g_e \right) \sum_{\vecs' \in \statespace'} \prod_{e' \in E'} \exp\!\left( g'_{e'} \, \psi^{e'}(\vecs') \right) = \exp\!\left( \sum_{e \in E} \beta_e g_e \right) Z(\Mo'; \vecg'). \qedhere
\]
\end{proof}

\begin{example}[Parity Ising, q-state Delta Potts, and And Ising Models on Graphs]
\label{ex:graphmodequiv}
To illustrate the usefulness of hypergraphical models, we show how three commonly used models in mathematics and physics fit in this framework. Let $G = (V,E)$ be a simple graph (a simple hypergraph where all edges have size 2). For each of the following choices of spin value set and interaction function (where $e = \{i,j\} \in E$), we obtain a hypergraphical model $\Mo = (G, \{\phi^e\}_{e \in E})$:
\begin{enumerate}
    \item \textbf{Parity Ising:} spins value set $\{-1, +1\}$, \quad $\phi^e_{\mathrm{Par-Is}}(\vecs) = s_i s_j$.
    \item \textbf{q-state Delta Potts:} spins value set $\{0, 1, \dots, q-1\}$, \quad $\phi^e_{\mathrm{q-Potts}}(\vecs) = \delta(s_i, s_j)$.
    \item \textbf{And Ising:} spins value set $\{0, 1\}$, \quad $\phi^e_{\mathrm{And-Is}}(\vecs) = s_i s_j$.
\end{enumerate}
We denote the corresponding models by $\Mo_{\mathrm{Par-Is}}$, $\Mo_{\mathrm{q-Potts}}$, and $\Mo_{\mathrm{And-Is}}$, respectively.

It holds that $\Mo_{\mathrm{Par-Is}} \cong \Mo_{\mathrm{2-Potts}}$. The state set bijection $\tau \colon \{0,1\} \to \{-1,+1\}$ given by $\tau(s) = 2s - 1$ transforms the interactions as follows:
\[ \tau(s_i)\tau(s_j) = (2s_i - 1)(2s_j - 1) = 2\delta(s_i, s_j) - 1. \]
This gives the required linear relation $\phi^e_{\mathrm{2-Potts}}(\vecs) = \frac{1}{2} \phi^e_{\mathrm{Par-Is}}(\hat{\tau}(\vecs)) + \frac{1}{2}$, showing $\Mo_{\mathrm{2-Potts}} \cong \Mo_{\mathrm{Par-Is}}$. This recovers the well-known relation between the partition functions of the Ising and Potts model given in equation \eqref{eq:introequiv}.

However, $\Mo_{\mathrm{And-Is}} \not \cong \Mo_{\mathrm{Par-Is}}$ (and thus $\Mo_{\mathrm{And-Is}} \not \cong \Mo_{\mathrm{2-Potts}}$ as well). To prove this, we analyze the level sets of the interaction functions on a single edge $e=\{i,j\}$. The Parity Ising interaction $\phi^e_{\mathrm{Par-Is}}$ takes the value $+1$ for two configurations ($s_i=s_j$) and $-1$ for two configurations ($s_i \neq s_j$). In contrast, the And interaction $\phi^e_{\mathrm{And-Is}}$ takes the value $1$ only for a single configuration $(1,1)$, and the value $0$ for the remaining three configurations. For the models to be isomorphic, there must exist an invertible linear transformation relating the interaction functions, which would necessarily preserve the cardinalities of these level sets. The discrepancy in cardinalities ($(2,2)$ versus $(1,3)$) proves that no such isomorphism exists.
\end{example}

The following proposition shows that duplicate and empty hyperedges are redundant for the maximum-entropy distribution defined by the model: their effect can always be absorbed into the parameters of a reduced model.
\begin{proposition}\label{prop:edge_redundancy}
Let $G = (V,E)$ be a hypergraph with $E = \{e_1, e_2, \dots, e_m\}$. Let $\Mo = (G, \{\phi^e \}_{e \in E})$ be a hypergraphical model associated with $G$, and let $p(\vecs \mid \Mo, \vecg)$ be the corresponding maximum-entropy distribution for parameter vector $\vecg$.

\begin{enumerate}[(i)]
    \item \textbf{(Duplicate Edges)} Suppose $e_i = e_j$ for some $i \neq j$. Without loss of generality, let $e_1 = e_2$. Consider the hypergraph $G' = (V, E')$ where $E' = \{e_2, e_3, \dots, e_m\}$, and let $\Mo' = (G', \{ \phi^e \}_{e \in E'})$. Then we have:
    \begin{equation}\label{eq:prop:double-edges:1}
        p\left(\vecs \mid \Mo, (g_{e_1}, g_{e_2}, g_{e_3} \dots, g_{e_m})\right) = p\left(\vecs \mid \Mo', (g_{e_1} + g_{e_2}, g_{e_3}, \dots, g_{e_m}) \right).
    \end{equation}

    \item \textbf{(Empty Edges)} Suppose $e_i = \emptyset$ for some $i$, i.e. an empty edge. Without loss of generality, let $e_1 = \emptyset$. Consider the hypergraph $G' = (V, E')$ where $E' = \{e_2, e_3, \dots, e_m\}$. Let $\Mo' = (G', \{ \phi^e\}_{e \in E'})$. Then we have:
    \begin{equation}\label{eq:prop:double-edges:2}
        p(\vecs \mid \Mo, (g_{e_1}, g_{e_2}, \dots, g_{e_m})) = p(\vecs \mid \Mo', (g_{e_2}, \dots, g_{e_m})).
    \end{equation}
\end{enumerate}
\end{proposition}

\begin{proof}
The equivalences in \eqref{eq:prop:double-edges:1} and \eqref{eq:prop:double-edges:2} 
follow directly from Definition \ref{def:maxent} and the locality conditions of \ref{def:graphmod}.
\end{proof}

This phenomenon of duplicate hyperedges being absorbed into the parameters is an analog of the same phenomenon for duplicate edges in graphs and the multivariable Tutte polynomial. See \cite{TRALDI2000} and the references therein.

\subsection{Interaction families}
We now introduce interaction families, which offer us a systematic way to induce hypergraphical models on hypergraphs.

\begin{definition}[Interaction Family]
\label{def:family}
An \emph{interaction family} $\IFam = (S, \{\Phi^k\}_{k \in \mathbb{Z}_{\geq 0}})$ is a \sset $S$ and a family of functions $\{\Phi^k\}_{k \in \mathbb{Z}_{\geq 0}}$, where for each $k$, the function $\Phi^k \colon S^k \to \mathbb{C}$ satisfies the following condition
\begin{enumerate}[(1)]
    \item $\Phi^k(\vecs) = \Phi^k(\vecs') \quad \text{ if } \{s_i \}_{i=1}^k = \{s'_i \}_{i=1}^k$ as multisets, i.e. $\Phi^k(\vecs)$ is invariant under permutation of the vertices (see remark \ref{rem:symmetry}).
\end{enumerate}
\end{definition}

\begin{definition}[Induced Models]
\label{def:modinst}
Given an interaction family $\IFam = (S, \{\Phi^k\}_{k \in \mathbb{Z}_{\geq 0}})$ and a hypergraph $G=(V,E)$, the \emph{hypergraphical model induced by $\IFam$ on $G$}, denoted $\Mo(G; \IFam)$, is the model $(G, \{\phi^e\}_{e \in E})$ with state space $S^{|V|}$. For every edge $e \in E$, the interaction function is defined as:
\[
    \phi^e(\vecs) = \Phi^{|e|}(\vecs|_e).
\]
Here, $\vecs|_e$ denotes the restriction of the state vector to the edge $e$. If $e = \{v_1, v_2, \dots, v_k\}$ is a multiset of vertices (where vertices may be repeated), then $\vecs|_e$ is the multiset of values $(s_{v_1}, s_{v_2}, \dots, s_{v_k})$.

For any class of hypergraphs $\mathcal{G}$, $\IFam$ induces a class of hypergraphical models $\Mo(\mathcal{G}; \IFam)$.
\end{definition}
An important class of hypergraphs is the class of loopless hypergraphs, denoted by $\overline{\mathcal{G}}$.

Finally, we define when two interaction families are isomorphic.

\begin{definition}[Isomorphism of Interaction Families]
\label{def:modclassiso}
Two interaction families $\IFam = (S, \{\Phi^k\}_{k \in \mathbb{Z}_{\geq 0}})$ and $\IFam' = (S', \{\Psi^k\}_{k \in \mathbb{Z}_{\geq 0}})$ are \emph{isomorphic} if there is a bijection $f \colon S \to S'$ such that, for every $k \in \mathbb{Z}_{\geq 0}$, the interaction functions $\Phi^k$ and $\Psi^k$ are linearly related; that is, there exist constants $\alpha_k, \beta_k$ such that for $\vecs = (s_i)_{i=1}^k \in S^k$ it holds that
\[\Phi^k\Big((s_i)_{i=1}^k\Big) = \alpha_k \Psi^k\Big(\big(f(s_i)\big)_{i=1}^k\Big) + \beta_k.\]
\end{definition}

\begin{proposition}
    \label{prop:isoinstant}
    Let $\IFam$ and $\IFam'$ be isomorphic interaction families, then, for any hypergraph $G = (V,E)$ it holds that $\Mo(G; \IFam) \cong \Mo(G; \IFam')$ with linear parameters $\alpha_e = \alpha_{|e|}$ and $\beta_e = \beta_{|e|}$.
\end{proposition}
\begin{proof}
    The hypergraph isomorphism is trivial in this case and the linear relation between the interaction functions follows directly from Definition \ref{def:modclassiso} and Definition \ref{def:modinst}.
\end{proof}

\begin{example}[Higher-Order Vector Potts Interaction Family]
    \label{ex:vectorpotts}
    The higher-order vector Potts model studied in~\cite{deClercq2026}
    fits into our framework as the interaction family defined as
    follows. For a given $q \in \mathbb{Z}_{>0}$, take \sset $S = \{
    \omega^0, \dots, \omega^{q-1} \} \subset \mathbb{C}$, the set of
    $q$-th roots of unity, where $\omega = e^{2\pi i / q}$, and
    interaction functions
    \[
        \Phi^k(s_1, \dots, s_k) = \prod_{i=1}^k s_i.
    \]
    These are symmetric under permutation of their arguments
    (Remark~\ref{rem:symmetry}), so $(S, \{\Phi^k\}_{k \in
    \mathbb{Z}_{\geq 0}})$ is an interaction family in the sense of
    Definition~\ref{def:family}. The induced hypergraphical models
    $\Mo(G; \IFam)$ and their maximum-entropy distributions
    coincide with the family of distributions studied
    in~\cite{deClercq2026}.
\end{example}

\subsection{Boolean interaction families and examples}
\label{sec:boolean_and_examples}

In this section, we focus on a subclass of models where the interaction values are binary. We define boolean hypergraphical models and interaction families, prove a reduction theorem showing that many models are isomorphic to boolean ones, and finally introduce three key examples of interaction families.

\begin{definition}[Boolean Hypergraphical Models and Interaction Families]
    A hypergraphical model $\Mo = (G, \{\phi^e\}_{e \in E} )$ is called \emph{boolean} if, for every $e \in E$, the interaction function $\phi^e$ maps $\statespace \to \{0,1\}$.

    A \emph{boolean interaction family} is an interaction family $\IFam = (S, \{\Phi^k\}_{k \in \mathbb{Z}_{\geq 0}})$ where the family of interaction functions maps to binary values:
    \[ \Phi^k \colon S^k \to \{0,1\} \quad \text{for all } k \geq 0. \]
\end{definition}
\begin{remark}[Binary Reduction]
\label{rem:binary_reduction}
    We note that any hypergraphical model (or interaction family) where the interaction functions take values in spin value sets with two different values is isomorphic to a boolean one: let $\phi$ take value in $\{u,v\}$ with $u \neq v$, then we have the linear transformation $\phi(\vecs) \mapsto \frac{\phi(\vecs) - v}{u - v}$, which maps $v \to 0$ and $u \to 1$. Consequently, the theory we develop for boolean models applies generally to any model with binary-valued interactions.
\end{remark}

We now introduce three different interaction families that induce the models on graphs discussed in Example~\ref{ex:graphmodequiv}. The first interaction family we define is the Parity Ising interaction family.

\begin{definition}[Parity Ising Interaction Family]
    \label{def:parisclass}
    The Parity Ising interaction family $\mci = (S, \{\PhiIs^k\}_{k \in \mathbb{Z}_{\geq 0}})$ is defined by the \sset $S = \{-1, 1\}$ and interaction functions
    \[
    \PhiIs^k \colon S^k \to \{-1,1\} \colon (s_i)_{i=1}^k \mapsto \prod_{i=1}^k s_i.
    \]
\end{definition}

The Parity Ising interaction family is not boolean, but we can use the insight from Remark~\ref{rem:binary_reduction} to find an isomorphic boolean interaction family.

\begin{definition}
    The boolean Parity Ising interaction family $\mcib = (S, \{\PhibIs^k \}_{k \in \mathbb{Z}_{\geq 0}})$ is defined with \sset $S = \mathbb{Z}/2\mathbb{Z}$ with interaction function
    \[ \PhibIs^k \colon S^k \to \{0,1\} \colon (s_i)_{i=1}^k \mapsto \begin{cases}
        0 & \sum_{i=1}^k s_i \equiv 1 \mod 2 \\
        1 & \text{else}
    \end{cases}.\]
\end{definition}
\begin{proposition}
\label{prop:parisboiso}
    The Parity Ising interaction family $\mci = (S, \{\PhiIs^k\}_{k \in \mathbb{Z}_{\geq 0}})$ is isomorphic to $\mcib = (S', \{\PhibIs^k\}_{k \in \mathbb{Z}_{\geq 0}})$.
\end{proposition}
\begin{proof}
    Define the bijection $f \colon S \to S' \colon s \mapsto \frac{1-s}{2}$, which maps $1 \mapsto 0$ and $-1 \mapsto 1$. Under this bijection, $s_i = (-1)^{f(s_i)}$, so
    \[
    \prod_{i=1}^k s_i = \prod_{i=1}^k (-1)^{f(s_i)} = (-1)^{\sum_{i=1}^k f(s_i)}.
    \]
    This equals $+1$ if $\sum_{i=1}^k f(s_i) \equiv 0 \mod{2}$ and $-1$ otherwise. Comparing with the definition of $\PhibIs^k$, we obtain
    \[
    \prod_{i=1}^k s_i = 2\,\PhibIs^k((f(s_i))_{i=1}^k) - 1,
    \]
    so that
    \[
    \PhiIs^k((s_i)_{i=1}^k) = 2\,\PhibIs^k((f(s_i))_{i=1}^k) - 1.
    \]
    This is the required linear relation from Definition~\ref{def:modclassiso} with $\alpha_k = 2$ and $\beta_k = -1$.
\end{proof}

\begin{definition}[Delta Potts Interaction Family]
    \label{def:delpottsfam}
    The Delta Potts interaction family $\mcp = (S, \{\PhiPotts^k\}_{k \in \mathbb{Z}_{\geq 0}})$, for a given integer $q > 0$, is defined by the \sset $S =[q]$ and interaction functions
    \[
    \PhiPotts^k \colon S^k \to \{0,1\} \colon \vecs \mapsto \delta_k(\vecs) \quad \text { where } \delta_k(\vecs) =  \begin{cases}
        1 & s_i = s_j \, \forall i,j \in [k] \\
        0 & \text{else}
    \end{cases}.
    \]
\end{definition}

In Example~\ref{ex:graphmodequiv}, we showed that for a graph $G$ and $q=2$, the Delta Potts model is isomorphic to the Parity Ising model: $\Mo(G; \mcpt) \cong \Mo(G; \mci)$. It is natural to ask whether this equivalence lifts to the level of interaction families, which would imply an isomorphism $\Mo(G; \mcpt) \cong \Mo(G; \mci)$ for all hypergraphs $G$. The following proposition shows that this is not the case: the two interaction families are fundamentally distinct, and their equivalence on graphs is a low-order coincidence that does not persist for higher-order interactions.

\begin{proposition}[Non-Isomorphism of Delta Potts and Parity Ising Interaction Families]
\label{prop:potts_parity_noniso}
The $q=2$ Delta Potts interaction family $\mcpt$ and the Parity Ising interaction family $\mci$ are not isomorphic: $\mcpt \not\cong \mci$.
\end{proposition}

\begin{proof}
An isomorphism of interaction families requires a linear relation $\Phi^k(\vecs) = \alpha_k \Psi^k(\pi(\vecs)) + \beta_k$ with $\alpha_k \neq 0$ for all $k$. Since $\alpha_k \neq 0$, this map is invertible and must therefore preserve the cardinalities of the level sets of the interaction functions. We show that these cardinalities differ for $k = 3$.

The Delta Potts interaction $\PhiPotts^3$ takes the value $1$ on $2$ configurations, $(0,0,0)$ and $(1,1,1)$, and $0$ on the remaining $6$ configurations. The Parity Ising interaction $\PhiIs^3$ takes the value $1$ on $4$ configurations (those with an even number of $-1$'s) and $-1$ on the other $4$ configurations. Since the level set cardinalities $(2, 6)$ and $(4, 4)$ do not match, no invertible linear transformation can relate the two interactions, and hence $\mcpt \not\cong \mci$.
\end{proof}

\begin{definition}[And Ising Interaction Family]
    The And Ising interaction family $\mca = (S, \{\PhiAnd^k\}_{k \in \mathbb{Z}_{\geq 0}})$ is defined by the \sset $S = \{0,1\}$ and interaction functions
    \[
    \PhiAnd^k \colon S^k \to \{0,1\} \colon (s_i)_{i=1}^k \mapsto \prod_{i=1}^k s_i.
    \]
\end{definition}

\subsection{Loopblind interaction families}

While the above interaction families are defined for multiset edges (loops), some interaction families are more naturally defined on loopless hypergraphs. To formalize this we define loopblind interaction families, which are interaction families whose functions on multisets can be reduced to another function in the interaction family on a set. This property will play an important role in Section~\ref{sec:delcon}, where we develop a theory of deletion--contraction for boolean interaction families: the loopblind condition ensures that the contraction of a hyperedge produces a canonical contracted model within the class of loopless hypergraphical models of that interaction family.

\begin{definition}[Loopblind Interaction Family]
    \label{def:loopless}
    A \emph{loopblind interaction family} is an interaction family $\IFam = (S, \{\Phi^k\}_{k \in \mathbb{Z}_{\geq 0}})$ that satisfies that for every $k \in \mathbb{Z}_{> 0}$ and every multisubset $A \subset [k]$ such that $\Phi^{|A|}$ is not constant, there is a unique \emph{subset} $A' \subseteq A$ such that $\Phi^{|A|}(\vecs|_A) = \Phi^{|A'|}(\vecs|_{A'})$ for all $\vecs \in S^k$.
\end{definition}

\begin{proposition}
    \label{prop:loopless_uniqueness}
    In Definition~\ref{def:loopless}, the uniqueness of $A'$ is automatic: if there exists any subset $A' \subseteq A$ satisfying $\Phi^{|A|}(\vecs|_A) = \Phi^{|A'|}(\vecs|_{A'})$ for all $\vecs \in S^k$ and $\Phi^{|A|}$ is not constant, then $A'$ is unique.
\end{proposition}
\begin{proof}
    Suppose $A'$ and $\widetilde{A}$ are two subsets of $A$ satisfying the condition. Then $\Phi^{|A'|}(\vecs|_{A'}) = \Phi^{|\widetilde{A}|}(\vecs|_{\widetilde{A}})$ for all $\vecs \in S^k$. If $A' \neq \widetilde{A}$, there exists an index $i \in A' \setminus \widetilde{A}$. Then $\Phi^{|A'|}(\vecs|_{A'})$ is independent of $s_i$, since the right-hand side does not depend on it. But $\Phi^{|A'|}$ is symmetric, so a function that is independent of one of its arguments must be independent of all of them, and hence constant. This contradicts the assumption that $\Phi^{|A|}$ is not constant.
\end{proof}

\begin{proposition}
    \label{prop:loopless_reduction}
    Let $\IFam$ be a loopblind interaction family. For any hypergraph $G = (V,E)$, there exists a loopless hypergraph $\overline{G} \in \overline{\mathcal{G}}$ such that $p(\vecs \mid \Mo(G; \IFam), \vecg) = p(\vecs \mid \Mo(\overline{G}; \IFam), \vecg)$ for all $\vecs$ and $\vecg$.
\end{proposition}
\begin{proof}
    Let $e \in E$ be a multiset hyperedge. By Definition~\ref{def:loopless}, if $\Phi^{|e|}$ is not constant, there is a unique subset $e' \subseteq e$ such that $\Phi^{|e|}(\vecs|_e) = \Phi^{|e'|}(\vecs|_{e'})$ for all $\vecs \in S^{|V|}$. If $\Phi^{|e|}$ is constant $e$ can be removed without altering the probabilities by the same logic as in the proof of Proposition~\ref{prop:edge_redundancy}. Replacing every multiset hyperedge by its corresponding subset yields a loopless hypergraph $\overline{G}$ with the same interaction functions, and therefore the same maximum-entropy distribution.
\end{proof}

\begin{proposition}
    The interaction families $\mci, \mcp$, and $\mca$ are loopblind.
\end{proposition}
\begin{proof}
    Let $k \in \mathbb{Z}_{>0}$, let $A$ be a multiset of elements in $[k]$, and let $\vecs \in S^k$. Let $m_i \geq 1$ denote the multiplicity of element $i$ in $A$, and let $A_{\mathrm{set}}$ be the underlying set of distinct elements in $A$.

    For the And Ising interaction family $\mca$, the interaction evaluates the product of the spins. Since $s_i \in \{0,1\}$, we have $s_i^{m_i} = s_i$. Therefore:
    \[ \Phi^{|A|}_{\mathrm{And-Is}}(\vecs|_A) = \prod_{i \in A} s_i = \prod_{i \in A_{\mathrm{set}}} s_i^{m_i} = \prod_{i \in A_{\mathrm{set}}} s_i = \Phi^{|A_{\mathrm{set}}|}_{\mathrm{And-Is}}(\vecs|_{A_{\mathrm{set}}}). \]
    Thus, the unique subset is $A' = A_{\mathrm{set}}$.

    For the Delta Potts interaction family $\mcp$, the interaction $\Phi^{|A|}_{\mathrm{q-Potts}}(\vecs|_A)$ equals $1$ if and only if all spins in the multiset $A$ are equal. Repeating an element does not change whether the set of elements shares a single value. So $\Phi^{|A|}_{\mathrm{q-Potts}}(\vecs|_A)$ equals $1$ if and only if all spins in the underlying set $A_{\mathrm{set}}$ are equal. Thus, $\Phi^{|A|}_{\mathrm{q-Potts}}(\vecs|_A) = \Phi^{|A_{\mathrm{set}}|}_{\mathrm{q-Potts}}(\vecs|_{A_{\mathrm{set}}})$, and the unique subset is again $A' = A_{\mathrm{set}}$.

    For the Parity Ising interaction family $\mci$ we can reduce the multisubsets by taking multiple occurences of a single vertex modulo $2$. Since $s_i \in \{-1,1\}$ repeating a spin $m_i$ times contributes $s_i^{m_i}$ to the product. Let $A_{\mathrm{odd}} = \{i \in A_{\mathrm{set}} \mid m_i \text{ is odd} \}$ be the subset of elements that appear an odd number of times in $A$. Because we only consider sums modulo $2$, we have:
    \[ \Phi^{|A|}_{\mathrm{Par-Is}}(\vecs|_A) = \prod_{i \in A} s_i = \prod_{i \in  A_{\mathrm{odd}}} s_i = \Phi^{|A_{\mathrm{odd}}|}_{\mathrm{Par-Is}}(\vecs|_{A_{\mathrm{odd}}}). \]
    Thus, the unique subset for the Parity Ising model is $A' = A_{\mathrm{odd}}$.
    Since a unique subset $A'$ exists for all three interaction families, they are all loopblind.
\end{proof}

\section{Deletion--contraction recurrence for loopblind boolean interaction families}
\label{sec:delcon}

In this section we show that the partition functions of models in boolean interaction families can be expressed in terms of an associated rank function. Then we give a set of sufficient conditions for the partition functions of boolean interaction families to satisfy a deletion--contraction recurrence. This thus achieves our original goal of establishing a `Tutte-polynomial-like' theory of deletion--contraction for an important type of interaction family.

\subsection{Rank generating function contraction}
To start, we introduce the rank function of a boolean hypergraphical model. We then show that the rank function contains all information in the model about its partition function.
\begin{definition}[Counting and Rank Function]
\label{def:rankfunc}
Let $\Mo = (G=(V,E),\{\phi^e\}_{e\in E})$ be a boolean hypergraphical model over $S$ with $|S| = q$.
We introduce its \emph{counting function}
\[
R_\Mo(A) = \#\{\vecs \in S^{|V|} \mid \phi^e(\vecs) = 1 \quad \forall e\in A\} \quad \text{for } A\subseteq E
\]
and the associated \emph{rank function} on $2^E$ by   
\[
r_\Mo(A) = \begin{cases} |V|- \log_q  R_\Mo(A) &\text{if } R_{\Mo}(A) > 0 \\
\infty &\text{else}
\end{cases} \quad \text{for } A\subseteq E. 
\] 
\end{definition}

We shall adopt the convention that, for $R_{\Mo}(A) = 0$, we have
\[ q^{-r_\Mo(A)} = q^{-\infty} = 0 = R_{\Mo}(A),\]
this will allow us to work with the rank function even if it is infinite, particularly in Section \ref{sec:polrank} where we examine when $r_\Mo(A)$ satisfies the necessary criteria to define a polymatroid.

\begin{example}[Rank Function of the Delta Potts Model on a Graph]
    \label{ex:rankfunc_potts_graph}
    Let $G = (V,E)$ be a simple graph and consider the Delta Potts model $\Mo(G; \mcp)$ from Definition~\ref{def:delpottsfam} with spin value set $S = [q]$. For a subset $A \subseteq E$, the counting function $R_{\Mo}(A)$ counts the number of configurations $\vecs \in S^{|V|}$ such that $\delta(s_i, s_j) = 1$ for all $\{i,j\} \in A$, i.e.\ the number of configurations in which $s_i = s_j$ whenever $\{i,j\} \in A$. This forces all spins within a connected component of the spanning subgraph $G_A = (V, A)$ to take the same value. Since there are $\kappa(G_A)$ connected components and $q$ choices per component, we obtain
    \[
    R_{\Mo}(A) = q^{\kappa(G_A)}.
    \]
    The rank function is therefore
    \[
    r_{\Mo}(A) = |V| - \log_q R_{\Mo}(A) = |V| - \kappa(G_A),
    \]
    which is precisely the rank function of the graphical matroid of $G$.
\end{example}

For a singleton $A = \{e\}$, we write $r_{\Mo}(e)$ instead of $r_{\Mo}(\{e\})$. Even though we have not (yet) associated any matroid or polymatroid to this rank function, we can nevertheless associate a formal polynomial to the rank function. We first define it abstractly, and then specialize to the rank function of a boolean hypergraphical model.
\begin{definition}[Rank generating function]
\label{def:abstract_rankpoly}
Let $E$ be a finite set and $r \colon 2^E \to \mathbb{R}_{\geq 0} \cup
\{\infty\}$ a function. The \emph{rank generating function}
of $r$ is
\[
    \widetilde Z(q, \vecv;\, r)
    \;:=\;
    \sum_{A \subseteq E} q^{-r(A)} \prod_{e \in A} v_e,
\]
where $q$ and $\vecv = (v_e)_{e \in E}$ are formal commuting variables,
with the convention $q^{-\infty} = 0$. When $r$ is integer-valued,
$\widetilde Z(q, \vecv;\, r)$ is a polynomial in $\vecv$ and $q^{-1}$.
\end{definition}

\begin{proposition}
    \label{prop:boolmodcountex}
    Let $\Mo = (G = (V,E), \{\phi^e\}_{e \in E} )$ be a boolean hypergraphical model. Then we can write the partition function as follows:
    \[
    Z(\Mo; \vecg) = q^{|V|} \widetilde{Z}(q, \vecv; r_\Mo),
    \]
    where $v_e = \exp(g_e) - 1$.
\end{proposition}
\begin{proof}
    We can write
    \begin{align*}
        Z(\Mo; \vecg) &= \sum_{\vecs \in \statespace} \prod_{e \in E} \exp(g_e \phi^e(\vecs)) \\
        &= \sum_{\vecs \in \statespace} \prod_{e \in E} (1 + (\exp(g_e) -1)\delta(\phi^e(\vecs),1)) \\
        &= \sum_{\vecs \in \statespace} \sum_{A \subseteq E} \prod_{e\in A} v_e \delta(\phi^e(\vecs),1) \\
        &= \sum_{A \subseteq E} \left( \prod_{e\in A} v_e \right) \left( \sum_{\vecs \in \statespace} \prod_{e\in A} \delta(\phi^e(\vecs),1) \right) \\
        &= \sum_{A \subseteq E} R_{\Mo}(A) \prod_{e \in A} v_e \\
        &= q^{|V|}\sum_{A \subseteq E} q^{-r_{\Mo}(A)} \prod_{e \in A} v_e = q^{|V|} \widetilde{Z}(q, \vecv; r_\Mo).
    \end{align*}
\end{proof} 
Proposition \ref{prop:boolmodcountex} shows that the partition function of a boolean model only depends on its counting function, which we formalize in the following corollary.
\begin{corollary}
\label{cor:rank_determines_partition}
Let $\Mo = (G=(V,E), \{\phi^e\}_{e \in E})$ and $\Mo' = (G'=(V',E),
\{\psi^e\}_{e \in E})$ be two boolean hypergraphical models over a common
edge set $E$ and a common spin value set size $q = |S| = |S'|$. Then
\[
    Z(\Mo; \vecg) = q^{\,|V| - |V'|}\, Z(\Mo'; \vecg)
    \qquad \text{for all } \vecg \in \mathbb{C}^{|E|}
\]
if and only if $r_\Mo = r_{\Mo'}$, that is, $r_\Mo(A) = r_{\Mo'}(A)$ for
every $A \subseteq E$.
\end{corollary}
\begin{proof}
By Proposition~\ref{prop:boolmodcountex}, $q^{-|V|} Z(\Mo; \vecg) =
\sum_{A \subseteq E} q^{-r_\Mo(A)} \prod_{e \in A} v_e$ with $v_e =
\exp(g_e) - 1$, and similarly for $\Mo'$. These are polynomials in the
$v_e$ whose coefficients are exactly the values $q^{-r_\Mo(A)}$, so they
agree for all $\vecg$ if and only if $r_\Mo$ and $r_{\Mo'}$ agree on
every $A \subseteq E$.
\end{proof}

The following theorem generalizes Tutte polynomial counting results to the interaction family setting.

\begin{theorem}[Counting Evaluations]
    \label{theo:counteval}
    Let $\Mo = (G=(V,E), \{\phi^e\}_{e \in E})$ be a boolean hypergraphical model. Then
    \begin{align}
        Z(\Mo; \vecg)\big|_{v_e = 1} &= \sum_{A \subseteq E} R_{\Mo}(A),\\
        Z(\Mo; \vecg)\big|_{v_e = -1} &= \#  \{ \vecs \in \statespace \mid \phi^e(\vecs)=0 \text{ for all } e \in E \}.
    \end{align}

\end{theorem}
\begin{proof}
    The first evaluation follows directly from the proof of Proposition \ref{prop:boolmodcountex}.

    For the second evaluation we note that
    \begin{align*}
         \sum_{A \subseteq E} R_{\Mo}(A) \prod_{e \in A} (-1) &= \sum_{A \subseteq E} (-1)^{|A|} R_{\Mo}(A) \\
         &= \sum_{\vecs \in \statespace} \sum_{A \subseteq E} \prod_{e \in A} (-\delta(\phi^e(\vecs),1)) \\
         &= \sum_{\vecs \in \statespace} \prod_{e \in E}(1- \delta(\phi^e(\vecs),1)) \\
         &= \sum_{\vecs \in \statespace} \prod_{e\in E} \delta(\phi^e(\vecs),0) \\
         &= \# \{ \vecs \in \statespace \mid \phi^e(\vecs)=0 \text{ for all } e \in E \}.
    \end{align*}
\end{proof}

\begin{definition}[Set Function Deletion and Contraction]
    Let $r \colon 2^E \to \mathbb{R}$ be a function and let $e \in E$. The \emph{deletion of} $e$ \emph{from} $r$ is the function $r \backslash e \colon 2^{E \setminus \{e\}} \to \mathbb{R}$ defined by
    \[
    r\backslash e(A) = r(A) \quad \text{for } A \subseteq E \setminus \{e\},
    \]
    and the \emph{contraction of} $e$ \emph{from} $r$ is the function $r/e \colon 2^{E \setminus \{e\}} \to \mathbb{R}$ defined by
    \[
    r/e(A) = r(A \cup \{e\}) \quad \text{for } A \subseteq E \setminus \{e\}.
    \]
\end{definition}

\begin{theorem}
    \label{theo:rankdelcon}
    Let $r \colon 2^E \to \mathbb{R}$, then
    \[
    \widetilde{Z}(q, \vecv; r) =  \widetilde{Z}(q,\vecv';  r\setminus e) + v_e \widetilde{Z}(q, \vecv' ; r/e),
    \]
    where $\vecv' = (v_f)_{f \in E \setminus \{ e\}
    }$.
\end{theorem}
\begin{proof}
    This follows directly by splitting the subset sum in the rank generating function as follows:
    \begin{align*}
        \sum_{A \subseteq E} q^{-r(A)} \prod_{f \in A} v_f &= \sum_{\substack{A \subseteq E \\ e \notin A}}q^{-r(A)} \prod_{f \in A} v_f + \sum_{\substack{A \subseteq E \\ e \in A}}q^{-r(A)} \prod_{f \in A} v_f \\
        &= \sum_{A \subseteq E \backslash \{e\}}q^{-r(A)} \prod_{f \in A} v_f + v_e \sum_{A \subseteq E \setminus \{e\}}q^{-r(A \cup \{e\})} \prod_{f \in A} v_f \\
        &= \sum_{A \subseteq E \backslash \{e \}} q^{- r\setminus e(A)} \prod_{f \in A} v_f + v_e \sum_{A \subseteq E \backslash \{e \}} q^{- r/e(A)} \prod_{f \in A} v_f \\
        &=\widetilde{Z}(q, \vecv';  r\setminus e) + v_e \widetilde{Z}(q, \vecv' ; r/e).
    \end{align*}
\end{proof}

\subsection{Loopblind boolean interaction family contraction}

Theorem~\ref{theo:rankdelcon} raises the question of when a contracted rank function $r_{\Mo} / e$ corresponds to the rank function $r_{\Mo/e}$ of a model $\Mo / e$ on a contracted hypergraph. We now introduce sufficient conditions on a boolean interaction family for this to hold. We shall restrict our attention to \emph{loopblind boolean interaction families} so that we take all hypergraphs in this section in the class of loopless hypergraphs $\overline{\mathcal{G}}$.

We begin by fixing some useful auxiliary notation. For a boolean interaction function $\Phi^k \colon S^k \to \{0,1\}$, we write
\[
    K_{\Phi^k} := \{\vecs \in S^k \mid \Phi^k(\vecs) = 1\}
\]
for its \emph{solution set}. Similarly, for a boolean hypergraphical model $\Mo(G; \IFam)$ and a hyperedge $e \in E$, we write
\[
    K_{\phi^e} := \{\vecs \in S^{V} \mid \phi^e(\vecs) = 1\} = \{\vecs \in S^{V} \mid \Phi^{|e|}(\vecs|_e) = 1\}
\]
for solution set of the constraint imposed by $e$. For a subset $A \subseteq E$ we write
\[
K_A = \bigcap_{e\in A} K_{\phi^e}
\]
for the solution set of the constraints imposed by $A$.

Before turning to the technical setup, let us sketch the structure of the argument we are aiming at, which motivates the definitions that follow. Fix a model $\Mo = \Mo(G; \IFam)$ on a loopless hypergraph $G$ and a hyperedge $e \in E$. Given a subset $A \subseteq E \setminus \{e\}$, we want to express the counting function $R_{\Mo}(A \cup \{e\})$ as the counting function $R_{\Mo'}(A')$ of a \emph{smaller} model $\Mo'$ on a contracted hypergraph. Three steps are needed:
\begin{enumerate}
    \item \textbf{Isolate the $e$-constraint.} Writing $R_{\Mo}(A \cup \{e\})$ with $\delta$-functions,
    \[
        R_{\Mo}(A \cup \{e\})
        = \sum_{\vecs \in S^V} \delta\big(\phi^e(\vecs),\, 1\big) \prod_{f \in A} \delta\big(\phi^f(\vecs),\, 1\big),
    \]
    the first factor $\delta(\phi^e(\vecs), 1)$ vanishes unless $\vecs \in K_{\phi^e}$. The sum over $S^V$ therefore collapses to a sum over the solution set of $e$:
    \[
        R_{\Mo}(A \cup \{e\}) = \sum_{\vecs \in K_{\phi^e}} \prod_{f \in A} \delta\big(\phi^f(\vecs),\, 1\big).
    \]
    In other words, satisfying the constraint on $e$ has restricted the state space from $S^V$ to $K_{\phi^e}$. To turn this into a model on a \emph{smaller} vertex set, we need to rewrite $K_{\phi^e}$ as a free state space over fewer vertices, and then check that the remaining constraints $\phi^f$ for $f \in A$ still correspond to interactions of $\IFam$ on that smaller set.

    \item \textbf{Reparameterise $K_{\phi^e}$ as a free state space.} We need $K_{\phi^e}$ to be in bijection with $S^{V \setminus D}$ for some subset $D \subseteq e$, so that summing over $K_{\phi^e}$ becomes summing freely over the remaining vertices. This is the role of \emph{functional representability} (Definition~\ref{def:funcrepr}): the solution set $K_{\Phi^k}$ is parameterised by a partition $[k] = I \sqcup D$ together with a dependency map $\sigma \colon S^I \to S^D$, so that each configuration of the independent variables $\vecs|_I$ uniquely determines $\vecs|_D = \sigma(\vecs|_I)$. Applied to $e$, this gives a bijection $S^{V \setminus D} \leftrightarrow K_{\phi^e}$, which we denote $\vecs \mapsto \hat\sigma(\vecs)$ — it extends $\vecs$ on $V \setminus D$ to all of $V$ by placing $\sigma(\vecs|_I)$ on $D$.

    Under this reparameterisation, each remaining interaction $\phi^f$ with $f \in A$ beco\-mes a \emph{restricted interaction}
    \[
        \widetilde\phi^f(\vecs) := \phi^f\!\big(\hat\sigma(\vecs)|_f\big),
    \]
    depending only on the free variables on $V \setminus D$.

    \item \textbf{Keep the restricted interactions inside the interaction family.} For the sum over $S^{V \setminus D}$ to count configurations of a model in $\IFam$, each restricted interaction $\widetilde\phi^f$ must itself be of the form $\Phi^{|f^*|}(\cdot|_{f^*})$ for some unique subset $f^* \subseteq V \setminus D$. We remark here that the uniqueness of $f^*$ is a consequence of $G$ being loopless, which is one reason why we restrict the family of loopless hypergraphs. This is the role of \emph{contraction-closedness} (Definition~\ref{def:contraction_closed}): substituting $\sigma$ into $\Phi^{|f|}$ produces an interaction of $\IFam$ on a subset of the remaining vertices. The resulting hyperedges $\{f^* \mid f \in A\}$ assemble into the contracted hypergraph $G /_{\IFam, \sigma} e$, and at that point
    \[
        R_{\Mo}(A \cup \{e\}) = R_{\Mo /_{\IFam, \sigma} e}(A),
    \]
    which is the rank-contraction identity we wanted.
\end{enumerate}

The rest of the section makes each of these steps precise. The technical statements live in Proposition~\ref{prop:modelfuncrep} (step~2), Proposition~\ref{prop:restricted_in_class} (step~3), and Theorem~\ref{theo:contraction_closure} (the assembled argument).

\subsubsection{Functional representability}
\begin{definition}
\label{def:funcrepr}
A loopblind boolean interaction family $\IFam = (S, \{\Phi^k\}_{k \in \mathbb{Z}_{\geq 0}})$ is \emph{functionally representable} if, for every $k$ there is a partition $[k] = I \sqcup D$ and a dependency map $\sigma \colon S^{|I|} \to S^{|D|}$ such that 
\[
K_{\Phi^k} = \{ (\vecs_I, \sigma(\vecs_I)) \mid \vecs_I \in S^{|I|} \}. 
\]
\end{definition}

\begin{example}[Examples of Functional Representability]
\label{ex:func_repr}
    We illustrate functional representability with a positive and a negative example.
    
    \textbf{Positive example.} The boolean Parity Ising interaction family $\mcib$ with $S = \mathbb{Z}/2\mathbb{Z}$ is functionally representable. For $k \geq 2$, the solution set consists of all even-weight configurations:
    \[
    K_{\PhibIs^k} = \left\{ (s_i)_{i=1}^k \in S^k \;\middle|\; \sum_{i=1}^k s_i \equiv 0 \mod{2} \right\},
    \]
    which has cardinality $|K_{\PhibIs^k}| = 2^{k-1}$. Setting $I = \{1, \ldots, k-1\}$, $D = \{k\}$, and $\sigma(\vecs_I) \equiv \sum_{i=1}^{k-1} s_i \mod{2}$ gives the required decomposition
    \[
    K_{\PhibIs^k} = \{ (\vecs_I, \sigma(\vecs_I)) \mid \vecs_I \in S^{|I|} \}.
    \]

    \textbf{Negative example.} Define the boolean \emph{OR interaction family} $\IFam^{\mathrm{Or}} = (S, \{\Phi_{\mathrm{Or}}^k\}_{k \in \mathbb{Z}_{\geq 0}})$ with $S = \{0,1\}$ and $\Phi_{\mathrm{Or}}^k(\vecs) = 1$ if and only if $\sum_{i=1}^k s_i \geq 1$ (i.e., at least one spin equals $1$).
    
    This interaction family is not functionally representable. The solution set has cardinality $|K_{\Phi^k_{\mathrm{Or}}}| = 2^k - 1$. For a dependency map $\sigma \colon S^{|I|} \to S^{|D|}$ to exist, the cardinality must equal $|S|^{|I|} = 2^{|I|}$. Since $2^k - 1$ is not a power of $2$ for $k > 1$, no such decomposition exists.
\end{example}

\begin{proposition}[Uniqueness of the Dependency Map]
\label{prop:func_decomp_equiv}
Let $\IFam = (S, \{\Phi^k\}_{k \in \mathbb{Z}_{\geq 0}})$ be a
functionally representable boolean interaction family, and let
$[k] = I \sqcup D$ with dependency map $\sigma$ be a decomposition
witnessing functional representability of $\Phi^k$. Then:
\begin{enumerate}
    \item $\sigma$ has the form $\sigma(\vecs_I) = (f(\vecs_I),
    \ldots, f(\vecs_I))$ for a symmetric function $f \colon S^{|I|}
    \to S$.
    \item For any other decomposition $[k] = I' \sqcup D'$ with
    dependency map $\sigma'$,
    we have $|I| = |I'|$ and $\sigma = \sigma'$ as functions $S^{|I|}
    \to S^{|D|}$ (under the canonical identification of $I$ with
    $I'$).
\end{enumerate}
\end{proposition}
\begin{proof}
\textbf{(1)} Symmetry of $\Phi^k$ implies that any permutation
applied independently to $I$-coordinates or $D$-coordinates of an
element of $K_{\Phi^k}$ stays in $K_{\Phi^k}$. Combined with
uniqueness of the dependency map, this gives $\sigma(\pi \cdot
\vecs_I) = \sigma(\vecs_I)$ for any $\pi \in S_{|I|}$ (so $\sigma$
depends only on the multiset of its inputs) and $\tau \cdot
\sigma(\vecs_I) = \sigma(\vecs_I)$ for any permutation $\tau$ of $D$
(so all components of $\sigma(\vecs_I)$ are equal). Writing
$f(\vecs_I)$ for this common value gives the claimed form.

\textbf{(2)} Functional representability gives $|K_{\Phi^k}| =
|S|^{|I|} = |S|^{|I'|}$, so $|I| = |I'|$. Pick a permutation $\pi
\in S_k$ with $\pi(I) = I'$ and $\pi(D) = D'$. Symmetry of $\Phi^k$
sends $(\vecs_I, \sigma(\vecs_I)) \in K_{\Phi^k}$ to $\pi \cdot
(\vecs_I, \sigma(\vecs_I)) \in K_{\Phi^k}$, and the second
decomposition then forces $\sigma'(\pi \cdot \vecs_I) = \pi \cdot
\sigma(\vecs_I)$. By (1), both sides simplify: $\sigma$ is invariant
under input permutations and has all output components equal, so
$\sigma'(\vecs_I) = \sigma(\vecs_I)$ for every $\vecs_I$.
\end{proof}

\begin{proposition}
    \label{prop:modelfuncrep}
    Let $\IFam = (S, \{\Phi^k\}_{k \in \mathbb{Z}_{\geq 0}})$ be a functionally representable boolean interaction family and $\Mo(G=(V,E); \IFam)$ be a hypergraphical model induced by $\IFam$ on a loopless hypergraph $G$. Furthermore let $e \in E$ be a hyperedge. Then there is a partition $I \sqcup D = e$ such that
    \[
    K_{\phi^e} = \{ (\vecs_{V \backslash e}, \vecs_I, \sigma(\vecs_I)) \mid \vecs_{V \backslash e} \in S^{|V| - |e|},\vecs_I \in S^{|I|}\}.
    \]
\end{proposition}
\begin{proof}
    Since $|e| = k$, fix an enumeration of the vertices of $e$ and let $[k] = I' \sqcup D'$ be the partition from Definition~\ref{def:funcrepr} with dependency map $\sigma$. This enumeration identifies $I'$ and $D'$ with subsets $I \subseteq e$ and $D \subseteq e$, giving a partition $e = I \sqcup D$.

    By the locality condition of Definition~\ref{def:graphmod}, the interaction function $\phi^e(\vecs)$ depends only on $\vecs|_e$, and by Definition~\ref{def:modinst} we have $\phi^e(\vecs) = \Phi^{|e|}(\vecs|_e)$. Therefore $\phi^e(\vecs) = 1$ if and only if $\vecs|_e \in K_{\Phi^k}$, which by functional representability holds if and only if $\vecs_D = \sigma(\vecs_I)$. Since the spins on $V \setminus e$ are unconstrained by $\phi^e$, the result follows.
\end{proof}

\subsubsection{Contraction-closure}

\begin{definition}[Contraction-closed Interaction Family]
\label{def:contraction_closed}
A loopblind boolean interaction family $\IFam = (S, \{\Phi^k\}_{k \in \mathbb{Z}_{\geq 0}})$ that is functionally representable is \emph{contraction-closed} if, for every $k \in \mathbb{Z}_{\geq 0}$, every $k' \geq k$, and every decomposition $[k] = I \sqcup D$ with dependency map $\sigma \colon S^{|I|} \to S^{|D|}$ witnessing functional representability of $\Phi^k$, the following holds. Extend $\sigma$ to a map
\[
    \hat\sigma \colon S^{[k'] \setminus D} \longrightarrow S^{[k']},
    \qquad
    \hat\sigma(\vecs)\big|_{[k'] \setminus D} = \vecs,
    \qquad
    \hat\sigma(\vecs)\big|_{D} = \sigma(\vecs|_I).
\]
Then for every $A \subseteq [k']$, either $\Phi^{|A|}(\hat\sigma(\vecs)|_A)$ is constant in $\vecs$, or there is a unique $A' \subseteq [k'] \setminus D$ such that
\[
    \Phi^{|A|}\!\big(\hat\sigma(\vecs)|_A\big) \;=\; \Phi^{|A'|}\!\big(\vecs|_{A'}\big)
    \qquad \text{for all } \vecs \in S^{[k'] \setminus D}.
\]
\end{definition}

\begin{example}[Examples of Contraction-closure]
\label{ex:contraction_closure}
We illustrate contraction-closure with a positive and a negative example.

\textbf{Positive example.} The boolean Parity Ising interaction family $\mcib$ is con\-traction-closed. Consider $k = 3$, $k' = 4$, with $I = \{1,2\}$, $D = \{3\}$, and dependency map $\sigma(s_1, s_2) \equiv s_1 + s_2 \mod 2$. Then the extended map is
\[
    \hat\sigma(s_1, s_2, s_4) = (s_1,\, s_2,\, s_1 + s_2,\, s_4) \mod 2.
\]
Consider $A = \{1, 3, 4\}$. Substituting gives
\[
    \PhibIs^{3}\!\big(\hat\sigma(s_1, s_2, s_4)|_A\big) = \PhibIs^{3}(s_1,\, s_1 + s_2,\, s_4) = 1 \iff s_2 + s_4 \equiv 0 \mod 2,
\]
so $A' = \{2, 4\}$. See Proposition~\ref{prop:mcib_properties} for the general argument.

\textbf{Negative example.} Define a boolean interaction family $\IFam^{\mathrm{mix}} = (S, \{\Phi^k_{\mathrm{mix}}\}_{k \in \mathbb{Z}_{\geq 0}})$ with $S = \{0,1\}$ by setting $\Phi^1_{\mathrm{mix}}(s) = s$ and $\Phi^k_{\mathrm{mix}}(\vecs) = 1$ iff $\sum_{i=1}^{k} s_i \equiv 0 \pmod 2$ for $k \geq 2$. In words, $\Phi^1_{\mathrm{mix}}$ checks whether a spin equals~$1$, while $\Phi^{k \geq 2}_{\mathrm{mix}}$ checks even parity.

This family is functionally representable for every $k$: for $k = 1$, we have $K_{\Phi^1_{\mathrm{mix}}} = \{1\}$ with $I = \emptyset$, $D = \{1\}$, and $\sigma(\emptyset) = 1$; for $k \geq 2$, the even-parity constraint gives $|K_{\Phi^k_{\mathrm{mix}}}| = 2^{k-1}$ with $I = [k-1]$, $D = \{k\}$, and $\sigma(\vecs_I) \equiv \sum_{i=1}^{k-1} s_i \pmod 2$.

However, contraction-closure fails already at $k = k' = 3$. Take $I = \{1,2\}$, $D = \{3\}$, $\sigma(s_1, s_2) \equiv s_1 + s_2 \pmod 2$, so that $\hat\sigma(s_1, s_2) = (s_1, s_2, s_1 + s_2)$, and $A = \{3\}$. Then
\[
    \Phi^1_{\mathrm{mix}}\!\big(\hat\sigma(s_1, s_2)|_A\big) \equiv s_1 + s_2 \mod 2,
\]
which equals $1$ iff $s_1 \neq s_2$. For contraction-closure this must equal $\Phi^{|A'|}_{\mathrm{mix}}(\cdot|_{A'})$ for some $A' \subseteq [k'] \setminus D = I$. But $A' = \{1,2\}$ gives $\Phi^2_{\mathrm{mix}}(s_1, s_2) = 1$ iff $s_1 = s_2$ (the complement); $A' = \{1\}$ or $A' = \{2\}$ gives $\Phi^1_{\mathrm{mix}}(s_i) = s_i$; and $A' = \emptyset$ is constant. None match, so no valid $A'$ exists.

Intuitively, the failure arises because $\Phi^1_{\mathrm{mix}}$ and $\Phi^{k \geq 2}_{\mathrm{mix}}$ encode different types of constraints: substituting the even-parity dependency map into $\Phi^1_{\mathrm{mix}}$ produces an odd-parity function, which does not belong to the family.
\end{example}

\begin{proposition}
\label{prop:contraction_closed_uniqueness}
In Definition~\ref{def:contraction_closed}, uniqueness of $A'$ is automatic: if there exists any subset $A' \subseteq [k'] \setminus D$ with $\Phi^{|A|}(\hat\sigma(\vecs)|_A) = \Phi^{|A'|}(\vecs|_{A'})$ for all $\vecs \in S^{[k'] \setminus D}$ and this function is not constant, then $A'$ is unique.
\end{proposition}
\begin{proof}
The proof follows the structure of Proposition~\ref{prop:loopless_uniqueness}: if two different subsets $A', \widetilde{A} \subseteq [k'] \setminus D$ both satisfy the condition, then there exists an index in $A' \setminus \widetilde{A}$, and symmetry of $\Phi^{|A'|}$ forces it to be constant, a contradiction.
\end{proof}

\begin{proposition}[Independence of Decomposition]\label{prop:independence_decomposition}
Contraction-closedness of $\IFam$ is independent of the choice of functionally representable decomposition: if the condition of Definition~\ref{def:contraction_closed} holds for one decomposition $[k] = I \sqcup D$ with dependency map $\sigma$, it holds for any other valid decomposition $[k] = I' \sqcup D'$ with dependency map $\sigma'$.
\end{proposition}
\begin{proof}
By Proposition~\ref{prop:func_decomp_equiv}, $|I| = |I'|$ and $\sigma = \sigma'$. Let $\pi \in S_k$ be a permutation with $\pi(I) = I'$ and $\pi(D) = D'$, extended by the identity on $[k'] \setminus [k]$ to a permutation of $[k']$ (still denoted $\pi$). Write $\hat\sigma$ and $\hat\sigma'$ for the extensions associated to the two decompositions.

Take any $A \subseteq [k']$. By contraction-closure with respect to the first decomposition applied to $\pi^{-1}(A)$, there is $B \subseteq [k'] \setminus D$ with
\[
    \Phi^{|A|}\!\big(\hat\sigma(\vecs)|_{\pi^{-1}(A)}\big) = \Phi^{|B|}(\vecs|_B)
    \qquad \text{for all } \vecs \in S^{[k'] \setminus D}.
\]
Since $\sigma = \sigma'$, $\Phi^{|A|}$ is symmetric, and $\pi$ is the identity on $[k'] \setminus [k]$, the left-hand side equals $\Phi^{|A|}(\hat\sigma'(\pi \cdot \vecs)|_A)$ after relabeling. Setting $A' := \pi(B) \subseteq [k'] \setminus D'$ gives
\[
    \Phi^{|A|}\!\big(\hat\sigma'(\vecs')|_A\big) = \Phi^{|A'|}(\vecs'|_{A'})
\]
for all $\vecs' \in S^{[k'] \setminus D'}$. Since $A \subseteq [k']$ was arbitrary, the condition holds for the second decomposition.
\end{proof}

\begin{definition}[Restricted Interaction Function]
\label{def:restricted_interaction}
Let $\IFam$ be a loopblind, functionally representable boolean interaction family, let $G = (V,E)$ be a loopless hypergraph, and let $e \in E$ have decomposition $e = I \sqcup D$ with dependency map $\sigma$. Extend $\sigma$ to a map
\[
    \hat\sigma \colon S^{V \setminus D} \longrightarrow S^{V},
    \qquad
    \hat\sigma(\vecs)\big|_{V \setminus D} = \vecs,
    \qquad
    \hat\sigma(\vecs)\big|_{D} = \sigma(\vecs|_I).
\]
For each $f \in E \setminus \{e\}$, the \emph{restricted interaction function} is
\[
    \widetilde\phi^f \colon S^{(f \cup I) \setminus D} \longrightarrow \{0,1\},
    \qquad
    \widetilde\phi^f(\vecs) \;:=\; \phi^f\!\big(\hat\sigma(\vecs)|_f\big),
\]
where, since $\phi^f$ depends only on the coordinates in $f$, and $\hat\sigma(\vecs)|_f$ is determined by $\vecs|_{(f \cup I) \setminus D}$, the right-hand side is well defined.
\end{definition}

\begin{proposition}[Restricted interactions lie in $\IFam$]
\label{prop:restricted_in_class}
Let $\IFam$ be a loopblind, functionally representable, contraction-closed boolean interaction family, let $G = (V,E)$ be a loopless hypergraph, let $e \in E$ have decomposition $e = I \sqcup D$ with dependency map $\sigma$, and let $f \in E \setminus \{e\}$. Then $\widetilde\phi^f$ is either constant, or there is a unique subset $f^* \subseteq (f \cup I) \setminus D$ such that
\[
    \widetilde\phi^f(\vecs) \;=\; \Phi^{|f^*|}(\vecs|_{f^*})
    \qquad \text{for all } \vecs \in S^{(f \cup I) \setminus D}.
\]
\end{proposition}
\begin{proof}
Set $k := |e|$ and $k' := |f \cup e|$, and fix an enumeration $[k'] \leftrightarrow f \cup e$ under which $e$ corresponds to $[k]$, the decomposition $I \sqcup D$ of $e$ corresponds to $[k] = I \sqcup D$ in Definition~\ref{def:contraction_closed}, and $f \setminus e$ corresponds to $[k'] \setminus [k]$. Under this identification, $f$ is a subset $A \subseteq [k']$ and $\widetilde\phi^f(\vecs) = \Phi^{|A|}(\hat\sigma(\vecs)|_A)$ is exactly the function appearing in Definition~\ref{def:contraction_closed}. Contraction-closedness then gives either constancy or a unique $A' \subseteq [k'] \setminus D$; taking $f^* \subseteq (f \cup I) \setminus D$ to be the image of $A'$ yields the claim.
\end{proof}
\begin{remark}[Role of loopblindness in uniqueness of $f^*$]
\label{rem:loopblind_uniqueness}
The loopblindness is necessary for the uniqueness in Proposition~\ref{prop:restricted_in_class}: for a hypergraph $G$ that is not loopless the uniqueness argument of Proposition~\ref{prop:contraction_closed_uniqueness} does not apply, and two distinct multisets $A_1, A_2$ could both reduce $\widetilde\phi^f$ to a valid expression in $\IFam$ without either being preferable.
\end{remark}

\begin{definition}[Hypergraph Deletion and Contraction with respect to a Loopblind Contraction--closed Interaction Family]
\label{def:induced_hyper_contraction}
Let $G=(V,E)$ be a loopless hypergraph, $\IFam$ a loopblind contraction-closed boolean interaction family, and $e \in E$ a hyperedge with decomposition $e = I \sqcup D$ and dependency map $\sigma$.

The \emph{deletion of $e$ from $G$} is the hypergraph $G \setminus e := (V, E \setminus \{e\})$.

The \emph{contraction of $e$ from $G$ with respect to $\IFam$ and $\sigma$}, denoted $G /_{\IFam, \sigma} e$, is the hypergraph $(V \setminus D,\, E')$ where $E' = \{ f^* \mid f \in E \setminus \{e\} \}$ is the multiset of \emph{contracted hyperedges} given, for each $f \in E \setminus \{e\}$, by Proposition~\ref{prop:restricted_in_class} applied to $\widetilde\phi^f$:
\begin{itemize}
    \item if $\widetilde\phi^f$ is constant, set $f^* := \emptyset$;
    \item otherwise, $f^*$ is the unique subset $f^* \subseteq (f \cup I) \setminus D$ with $\widetilde\phi^f = \Phi^{|f^*|}(\cdot|_{f^*})$.
\end{itemize}
\end{definition}
\begin{definition}[Model Deletion and Contraction]
\label{def:model_delcon}
Let $\Mo$ be the hypergraphical model induced by a loopblind interaction family $\IFam$ on a loopless hypergraph $G =(V,E)$ and let $e \in E$. We introduce the following shorthand notation:

The \emph{deletion} of $e$ from $\Mo$, denoted $\Mo \setminus e$, is the model induced by $\IFam$ on $G \setminus e$:
\[ \Mo \setminus e := \Mo(G \setminus e; \IFam). \]

If $\IFam$ is contraction-closed and $e$ has a dependency map $\sigma$, the \emph{contraction} of $e$ from $\Mo$, denoted $\Mo /_{\IFam,\sigma} e$, is the model induced by $\IFam$ on the contracted hypergraph (Definition \ref{def:induced_hyper_contraction}):
\[ \Mo /_{\IFam,\sigma} e := \Mo(G/_{\IFam, \sigma} e; \IFam). \]
\end{definition}

\begin{theorem}[Deletion--Contraction Recurrence for Contraction-closed Families]
\label{theo:contraction_closure}
Let $\IFam$ be a loopblind boolean interaction family that is functionally representable and contraction-closed. Let $\Mo = \Mo(G; \IFam)$ be the hypergraphical model induced by $\IFam$ on a loopless hypergraph $G = (V,E)$, and let $e \in E$ with decomposition $e = I \sqcup D$ and dependency map $\sigma$. Then, for all $A \subseteq E \setminus \{e\}$,
\begin{align}
    r_{\Mo}/e(A) &= |D| + r_{\Mo /_{\IFam,\sigma} e}(A), \\
    r_{\Mo} \setminus e(A) &= r_{\Mo \setminus e}(A).
\end{align}
\end{theorem}

\begin{proof}
The deletion identity is immediate from the definitions of rank deletion and model deletion. We focus on the contraction identity.

Set $V' := V \setminus D$ and $d := |D|$, and recall that rank contraction is defined by $r_{\Mo \, / \, e}(A) = r_{\Mo}(A \cup \{e\})$, which is determined by the counting function $R_{\Mo}(A \cup \{e\})$. Let $\hat\sigma \colon S^{V'} \to S^{V}$ be the extension of $\sigma$ from Definition~\ref{def:restricted_interaction}. By Proposition~\ref{prop:modelfuncrep}, a configuration $\vecs \in S^{V}$ satisfies $\phi^e(\vecs) = 1$ if and only if $\vecs = \hat\sigma(\vecs')$ for some $\vecs' \in S^{V'}$, and this correspondence is a bijection between $K_{\phi^e}$ and $S^{V'}$. Therefore,
\begin{align*}
    R_{\Mo}(A \cup \{e\})
    &= \sum_{\vecs \in S^{V}} \delta\big(\phi^e(\vecs),\, 1\big) \prod_{f \in A} \delta\big(\phi^f(\vecs),\, 1\big) \\
    &= \sum_{\vecs' \in S^{V'}} \prod_{f \in A} \delta\big(\phi^f(\hat\sigma(\vecs')),\, 1\big) \\
    &= \sum_{\vecs' \in S^{V'}} \prod_{f \in A} \delta\big(\widetilde\phi^f(\vecs'|_{(f \cup I) \setminus D}),\, 1\big),
\end{align*}
where the final equality uses Definition~\ref{def:restricted_interaction}. By Proposition~\ref{prop:restricted_in_class}, each $\widetilde\phi^f$ is either constantly $1$ (in which case $f^* = \emptyset$ and the corresponding factor is $1$) or equals $\Phi^{|f^*|}(\cdot|_{f^*})$ for a unique $f^* \subseteq V'$; in either case the factor equals $\psi^{f^*}(\vecs')$, the interaction function of the edge $f^*$ in the contracted model $\Mo /_{\IFam,\sigma} e$. Hence
\[
    R_{\Mo}(A \cup \{e\})
    = \sum_{\vecs' \in S^{V'}} \prod_{f^* \in A^*} \delta\big(\psi^{f^*}(\vecs'),\, 1\big)
    = R_{\Mo /_{\IFam,\sigma} e}(A),
\]
where $A^* := \{ f^* \mid f \in A \}$. Converting to rank functions using $|V| = |V'| + d$,
\[
    r_{\Mo \, / \, e}(A)
    = |V| - \log_q R_{\Mo}(A \cup \{e\})
    = (|V'| + d) - \log_q R_{\Mo /_{\IFam,\sigma} e}(A)
    = d + r_{\Mo /_{\IFam,\sigma} e}(A). \qedhere
\]
\end{proof}
\begin{corollary}[Partition Function Recurrence]
\label{cor:partition_delcon}
    Under the assumptions of Theorem \ref{theo:contraction_closure}, the partition function satisfies the following deletion-contraction recurrence:
    \[
    Z(\Mo; \vecg) = Z(\Mo \backslash e; \vecg') + v_e Z(\Mo/_{\IFam,\sigma}e; \vecg'),
    \]
    where $\vecg' = (g_f)_{f \in E \setminus \{e\}}$.
\end{corollary}

\begin{proof}
    We use the relation $Z(\Mo; \vecg) = q^{|V|} \widetilde{Z}(q, \vecv; r_{\Mo})$ from Proposition \ref{prop:boolmodcountex} and the rank recurrence from Theorem \ref{theo:rankdelcon}:
    \begin{align*}
        Z(\Mo; \vecg) &= q^{|V|} \left( \widetilde{Z}(q, \vecv'; r_{\Mo}\setminus e) + v_e \widetilde{Z}(q, \vecv'; r_{\Mo}/ e) \right) \\
        &= q^{|V|} \widetilde{Z}(q, \vecv'; r_{\Mo \setminus e}) + v_e q^{|V|} \sum_{A \subseteq E \setminus \{e\}} q^{-(d + r_{\Mo/_{\IFam, \sigma}e}(A))} \prod_{f \in A} v_f \\
        &= Z(\Mo \setminus e; \vecg') + v_e q^{|V|-d} \sum_{A \subseteq E \setminus \{e\}} q^{-r_{\Mo/_{\IFam, \sigma}e}(A)} \prod_{f \in A} v_f \\
        &= Z(\Mo \setminus e; \vecg') + v_e q^{|V'|} \widetilde{Z}(q, \vecv'; r_{\Mo/_{\IFam, \sigma}e}) \\
        &= Z(\Mo \setminus e; \vecg') + v_e Z(\Mo/_{\IFam,\sigma}e;\vecg').
    \end{align*}
\end{proof}

We show later that the three loopblind boolean interaction families introduced in Section \ref{sec:boolean_and_examples} are contraction--closed and functionally representable.

\section{Polymatroid rank functions}
\label{sec:polrank}
In this section we connect the rank functions of the interaction families defined in \ref{def:rankfunc} to the theory of polymatroids. We first show that the rank function of any hypergraphical model satisfies the normalization and monotonicity property of polymatroids. We then give two conditions on an interaction family and prove that these are sufficient for the rank functions of hypergraphical models induced by interaction families satisfying these conditions to be submodular. Furthermore, we demonstrate that if an interaction family is functionally representable and contraction-closed as well, its rank function is integer-valued. Together, these conditions yield Corollary~\ref{cor:rank_is_polymatroid}, which establishes that the rank function of any model induced by such an interaction family is a polymatroid. This achieves our goal of connecting the theory of boolean interaction families to the theory of polymatroids.

We begin with the definition of a polymatroid, following \cite{Vertigan_Whittle_1993}.
\begin{definition}
    Consider a pair $(E,r)$ where $E$ is a finite set, called the \emph{ground set}, and $r \colon 2^E \to \mathbb{R}_{\geq0}$ is a function, called the \emph{rank function}. If $r$ satisfies the following properties
    \begin{itemize}
        \item \textit{Normalization}: $r(\emptyset) = 0$.
        \item \textit{Monotonicity}: $r$ is increasing, meaning that $r(A) \leq r(B)$ whenever $A \subseteq B \subseteq E$.
        \item \textit{Submodularity}: $r$ satisfies the inequality
        \[
        r(A) + r(B) \geq r(A \cup B) + r(A \cap B)
        \]
        for all subsets $A, B \subseteq E$,
        \item \textit{Integer-valued}: $r(A) \in \mathbb{Z}_{\geq 0}$ for all $A \subseteq E$,
    \end{itemize}
    then we call the pair $(E,r)$ a \emph{polymatroid}. It is a \textit{$k$-polymatroid} if $r(e) \leq k$ for all $e \in E$. A $1$-polymatroid is called a matroid.
\end{definition}
\begin{remark}
We remark that different authors use different definition of polymatroids in the literature.
Some authors allow the rank function $r$ to take values in $\mathbb{R}$~\cite{bonin2023natural}, while other authors do not require $r$ to be increasing~\cite{kalman2013version,bernardi2022universal,guan2023deletion}.
We follow the definitions of~\cite{Vertigan_Whittle_1993} throughout.
\end{remark}
In the rest of the section we shall go through these properties. We show that normalization and monotonicity hold for the rank functions of all boolean hypergraphical models. Then we introduce the group--coset and global satisfiability conditions for classifying submodularity.

\subsection{Normalization and monotonicity}
Here we show that normalization and monotonicity hold for all rank functions of boolean hypergraphical models.

\begin{proposition}[Normalization]
    Let $\Mo = (G=(V,E), \{\phi^e\}_{e \in E})$ be a boolean hypergraphical model with \sset $S$. Then $r_{\Mo}(\emptyset) = 0$.
\end{proposition}
\begin{proof}
    Let $|S| = q$. By definition, the counting function for the empty set is the total number of configurations in the state space, as there are no constraints to satisfy. Thus:
    \[ R_{\Mo}(\emptyset) = \#\{ \vecs \in S^{|V|} \} = q^{|V|}. \]
    Substituting this into the definition of the rank function:
    \[ r_{\Mo}(\emptyset) = |V| - \log_q(R_{\Mo}(\emptyset)) = |V| - \log_q(q^{|V|}) = |V| - |V| = 0. \]
\end{proof}

Monotonicity also holds for all boolean hypergraphical models, since extra hyper\-edges can only restrict the solution space, thereby decreasing the count $R_{\Mo}$ and increasing the rank $r_{\Mo}$.

\begin{proposition}[Monotonicity]
    Let $\Mo = (G=(V,E), \{\phi^e\}_{e \in E})$ be a boolean hypergraphical model. The rank function $r_{\Mo}$ is monotonically increasing; that is, for any $A \subseteq B \subseteq E$, it holds that $r_{\Mo}(A) \leq r_{\Mo}(B)$.
\end{proposition}
\begin{proof}
    Recall $R_{\Mo}(A) = \#\{\vecs \in \statespace \mid \phi^e(\vecs) = 1 \quad \forall e \in A\}$.
    Let $A \subseteq B \subseteq E$. The set of configurations satisfying all constraints in $B$ is a subset of the set of configurations satisfying the constraints in $A$. Thus, $R_{\Mo}(B) \leq R_{\Mo}(A)$.
    
    If $R_{\Mo}(B) > 0$, then $R_{\Mo}(A) > 0$, and since the logarithm is a strictly increasing function:
    \[ r_{\Mo}(A) = |V| - \log_q R_{\Mo}(A) \leq |V| - \log_q R_{\Mo}(B) = r_{\Mo}(B). \]
    If $R_{\Mo}(B) = 0$, then by convention $r_{\Mo}(B) = \infty$, which is greater or equal to all possible values of $r_{\Mo}(A)$.
\end{proof}

\subsection{Submodularity}
The third polymatroid property is submodularity. Unlike normalization and monotonicity, submodularity is not guaranteed for arbitrary interaction functions. It requires specific algebraic structure in the solution sets. Here we introduce two conditions that are sufficient for submodularity.

\begin{definition}[Group-Coset]
\label{def:group_coset_class}
Let $\Mo=(G=(V,E), \{\phi^e\}_{e \in E})$ be a hypergraphical model, whose \sset $S$ is a finite abelian group. Then $\Mo$ is called \emph{group--coset} if for every $e \in E$, the solution set $K_{\phi^e}=\{\vecs\in S^{|V|} \mid \phi^e(\vecs)=1\}$ is a coset of a subgroup of $S^{|V|}$.

An interaction family $\IFam = (S,\{ \Phi^k\}_{k \in \mathbb{Z}_{\geq 0}})$, whose \sset $S$ is a finite group is \emph{group--coset} if, for all $k \in \mathbb{Z}_{\geq 0}$ it holds that $K_{\Phi^k}$ is a coset of a subgroup of $S^k$.
\end{definition}

\begin{definition}[Global Satisfiability]
\label{def:global_satisfiable}
A boolean hypergraphical model $\Mo$ with a \sset $S$ is \emph{globally satisfiable} if there exists an $s \in S$ such that, for $\vecs=(s, \dots, s) \in S^{|V|}$ we have $\phi^e(\vecs)=1$ for all $e \in E$. 

An interaction family $\IFam = (S,\{ \Phi^k\}_{k \in \mathbb{Z}_{\geq 0}})$ with a \sset $S$ is globally satisfiable if there is an $s \in S$ such that, for $\vecs^k = (s, \dots s) \in S^k$, it holds that $\Phi^k(\vecs^k)=1$ for all $k$.
\end{definition}

\begin{theorem}[Submodularity]
\label{thm:submodularity}
    Let $\Mo$ be a boolean hypergraphical model. If $\Mo$ is globally satisfiable and group--coset, then the rank function $r_{\Mo}$ is submodular:
    \[ r_{\Mo}(A) + r_{\Mo}(B) \geq r_{\Mo}(A \cup B) + r_{\Mo}(A \cap B) \quad \forall A, B \subseteq E. \]
\end{theorem}
\begin{proof}
    Recall that $K_A = \bigcap_{e \in A} K_{\phi^e}$ denotes the solution set for the constraints in $A$. 
    Since $\Mo$ is group--coset, for every $e \in E$, the set $K_{\phi^e}$ is a coset of a subgroup $H_e \leq S^{|V|}$.

    Crucially, since $\Mo$ is globally satisfiable (Definition \ref{def:global_satisfiable}), there exists a configuration $\vecs^{|V|} \in S^{|V|}$ such that $\phi^e(\vecs^{|V|}) = 1$ for all $e \in E$. 
    Therefore, $\vecs^{|V|} \in K_A$ for any $A \subseteq E$, which implies that the intersection $K_A$ is non-empty.

    By the standard property of cosets, the non-empty intersection of cosets is itself a coset of the intersection of the corresponding subgroups. 
    Thus, $K_A$ is a coset of the subgroup $H_A = \bigcap_{e \in A} H_e$.
    Since the size of a coset is equal to the size of its underlying subgroup, the counting function satisfies $R_{\Mo}(A) = |H_A|$.
    
    We now use the standard group isomorphism theorem relating the order of the intersection and sum of subgroups: 
    \[ |H_A| \cdot |H_B| = |H_A \cap H_B| \cdot |H_A + H_B|. \]
    Note that $H_{A \cup B} = H_A \cap H_B$. 
    Furthermore, $H_A + H_B$ is the subgroup generated by $H_A$ and $H_B$. Since $H_A, H_B \subseteq H_{A \cap B}$, it follows that $H_A + H_B \leq H_{A \cap B}$. 
    Consequently, $|H_A + H_B| \leq |H_{A \cap B}|$.
    
    Substituting these relations into the isomorphism identity:
    \[ |H_A| \cdot |H_B| \leq |H_{A \cup B}| \cdot |H_{A \cap B}|. \]
    Taking the logarithm base $q$ and negating (reversing the inequality) yields:
    \[ -\log_q R_{\Mo}(A) - \log_q R_{\Mo}(B) \geq -\log_q R_{\Mo}(A \cup B) - \log_q R_{\Mo}(A \cap B). \]
    Adding $2|V|$ to both sides yields the submodularity of the rank function.
\end{proof}

\subsection{Integer-valuedness}
The final property requires the rank function to take integer values. We shall introduce conditions for this to hold on the interaction family level. Firstly we extend global satisfiability to the interaction family level.

\begin{definition}
    Let $\IFam = (S, \{\Phi^k\}_{k \in \mathbb{Z}_{\geq 0}})$ be an interaction family. Then $\IFam$ is globally satisfiable if every hypergraphical model $\Mo(G; \IFam)$ induced by $\IFam$ is globally satisfiable.
\end{definition}
 
\begin{theorem}[Integer-valuedness]
\label{thm:structural_integer_rank}
Let $\IFam$ be a boolean interaction family that is globally satisfiable, functionally representable, and contraction-closed.
Then for any model $\Mo(G; \IFam)$ and subset $A \subseteq E$, the rank function $r_{\Mo}(A)$ is an integer.
\end{theorem}
\begin{proof}
We prove by induction on $n \in \mathbb{Z}_{\geq 0}$ the following statement:
\begin{quote}
    For every hypergraphical model $\Mo = \Mo(G; \IFam)$ induced by $\IFam$ on a loopless hypergraph $G = (V, E)$, and every subset $A \subseteq E$ with $|A| = n$, the value $r_{\Mo}(A)$ is an integer.
\end{quote}

\emph{Base case ($n = 0$).} Then $A = \emptyset$ and $R_{\Mo}(\emptyset) = q^{|V|}$, so $r_{\Mo}(\emptyset) = 0 \in \mathbb{Z}$.

\emph{Inductive step.} Suppose the statement holds for $n - 1$, and let $\Mo = \Mo(G; \IFam)$ and $A \subseteq E$ with $|A| = n$. Pick any $e \in A$ and set $A' := A \setminus \{e\}$, so $|A'| = n - 1$. Let $e = I \sqcup D$ be a decomposition with dependency map $\sigma$. By Theorem~\ref{theo:contraction_closure},
\[
    r_{\Mo}(A) = r_{\Mo}/e(A') = |D| + r_{\Mo /_{\IFam,\sigma} e}(A').
\]
The term $|D|$ is an integer. Contraction-closedness of $\IFam$ ensures that $\Mo /_{\IFam,\sigma} e$ is again a model induced by $\IFam$ on a loopless hypergraph, so the inductive hypothesis applied to this model and the subset $A'$ gives that $r_{\Mo /_{\IFam,\sigma} e}(A')$ is an integer. Hence $r_{\Mo}(A)$ is an integer, completing the induction.
\end{proof}

Combining the results of the previous subsections, we obtain
sufficient conditions on an interaction family ensuring that the
rank function of every induced hypergraphical model is a
polymatroid. This is the second main result announced in the
introduction.

\begin{corollary}[Polymatroid Rank]
    \label{cor:rank_is_polymatroid}
    Let $\IFam$ be a loopblind boolean interaction family
    that is group--coset, globally satisfiable, functionally
    representable, and contraction-closed. Then for every loopless
    hypergraph $G$, the rank function $r_{\Mo(G; \IFam)}$
    is a polymatroid.
\end{corollary}

\section{Application to three common interaction families}
\label{sec:example-interaction-families}
In this section, we demonstrate the usefulness of our `Tutte-polynomial-like' theory of boolean interaction families by applying it to the three commonly used boolean interaction families introduced in Section~\ref{sec:boolean_and_examples}: the Parity Ising interaction family, the Delta Potts interaction family, and the And Ising interaction family. We recover three types of polymatroids commonly associated to hypergraphs and show that a number of results from the literature naturally follow from our framework.

For each interaction family, we follow a systematic analysis:
\begin{enumerate}
    \item We verify the algebraic properties required for our deletion--contraction theory (functional representability and contraction-closure).
    \item We identify the specific combinatorial structure (matroid or polymatroid) that corresponds to the model's rank function.
    \item We derive the explicit, model-specific geometric operations for deletion and contraction on the hypergraph.
\end{enumerate}

We begin with a subsection introducing three polynomials commonly associated to (poly)matroids and record how our rank generating function specializes to each. The three subsequent subsections then apply the framework to the Parity Ising, Delta Potts, and And Ising families.

\subsection{Related polynomials}
\label{sec:poincare}

The rank generating function of Definition~\ref{def:abstract_rankpoly} is related to three classical polynomial invariants in the case that the rank function $r$ is a (poly)matroid: the Poincar\'e polynomial of a polymatroid, introduced in~\cite{Helg}; the $k$-polymatroid Tutte polynomial of~\cite{Berrekkal2026}; and the multivariate Tutte polynomial of a matroid~\cite{Sok05, ellis-monaghanHandbookTuttePolynomial2022}. We recall all three, and record the specializations of our rank generating function under which each is recovered. The three subsequent subsections identify each family's partition function as an evaluation of one or more of these polynomials on a specific polymatroid.

\begin{definition}[Poincar\'e polynomial~\cite{Helg}]
\label{def:Helgason_poincare}
Let $(E, r)$ be a polymatroid. Its \emph{Poincar\'e polynomial} is
\[
    \tau(\lambda, \eta;\, r)
    \;:=\;
    \sum_{A \subseteq E} (\eta - 1)^{|A|}\, \lambda^{r(E) - r(A)},
\]
where $\lambda$ and $\eta$ are formal commuting variables. The
associated \emph{characteristic polynomial} is
\[
    \tau(\lambda, 0;\, r)
    \;=\; \sum_{A \subseteq E} (-1)^{|A|}\, \lambda^{r(E) - r(A)}.
\]
\end{definition}

\begin{definition}[$k$-polymatroid Tutte polynomial~\cite{Berrekkal2026}]
\label{def:bem_Tk}
Let $(E, r)$ be a $k$-polymatroid. Its \emph{$k$-polymatroid Tutte
polynomial} is
\[
    T_k(P;\, x, y)
    \;:=\;
    \sum_{A \subseteq E} (x - 1)^{r(E) - r(A)}\, (y - 1)^{k|A| - r(A)},
\]
where $x$ and $y$ are formal commuting variables.
\end{definition}

\begin{definition}[Multivariate Tutte polynomial~\cite{Sok05}]
\label{def:multi_tutte}
Let $M = (E, r)$ be a matroid. Its \emph{multivariate Tutte
polynomial} is
\[
    \widetilde Z_{\mathrm{Tutte}}(q, \vecv;\, M)
    \;:=\;
    \sum_{A \subseteq E} q^{-r(A)} \prod_{e \in A} v_e,
\]
where $q$ and $\vecv = (v_e)_{e \in E}$ are formal commuting
variables.
\end{definition}

The three polynomials and our rank generating function are linked by the
following specializations.

\begin{proposition}[Rank Generating Function, Poincar\'e, $T_k$, and Tutte]
\label{prop:rankpoly_specializations}
Let $r \colon 2^E \to \mathbb{Z}_{\geq 0}$ be a function.
\begin{enumerate}
    \item If $(E, r)$ is a polymatroid, then
    \[
        \widetilde Z\bigl(q,\, (\eta - 1)_{e \in E};\, r\bigr)
        \;=\;
        q^{-r(E)}\, \tau(q, \eta;\, r).
    \]
    \item If $(E, r)$ is a $k$-polymatroid $P$, then for $v_e = v$
    uniformly across edges,
    \[
        \widetilde Z(q, v;\, r)
        \;=\;
        \Bigl(\tfrac{v^{1/k}}{q}\Bigr)^{r(E)}\,
        T_k\bigl(P;\, 1 + \tfrac{q}{v^{1/k}},\, 1 + v^{1/k}\bigr).
    \]
    \item If $(E, r)$ is a matroid $M$, then
    \[
        \widetilde Z(q, \vecv;\, r)
        \;=\;
        \widetilde Z_{\mathrm{Tutte}}(q, \vecv;\, M).
    \]
\end{enumerate}
\end{proposition}

\begin{proof}
For (1), substituting $v_e = \eta - 1$ in
Definition~\ref{def:abstract_rankpoly} gives
\[
    \widetilde Z\bigl(q,\, (\eta - 1)_{e \in E};\, r\bigr)
    \;=\;
    \sum_{A \subseteq E} q^{-r(A)}\, (\eta - 1)^{|A|}
    \;=\;
    q^{-r(E)} \sum_{A \subseteq E} (\eta - 1)^{|A|}\,
        q^{r(E) - r(A)},
\]
which is $q^{-r(E)} \tau(q, \eta;\, r)$ by
Definition~\ref{def:Helgason_poincare}.

For (2), factoring $(x-1)^{r(E)}$ out of the subset sum defining
$T_k$ gives
\begin{align*}
    T_k(P;\, x, y)
    \;&=\;
    (x-1)^{r(E)}
    \sum_{A \subseteq E} \bigl[(x-1)(y-1)\bigr]^{-r(A)}\,
        \bigl[(y-1)^k\bigr]^{|A|} \\
    \;&=\;
    (x-1)^{r(E)}\, \widetilde Z\bigl((x-1)(y-1),\, (y-1)^k;\, r\bigr).
\end{align*}
Substituting $x = 1 + q/v^{1/k}$ and $y = 1 + v^{1/k}$ yields
$(x-1)(y-1) = q$ and $(y-1)^k = v$, so the rank generating function on the
right-hand side equals $\widetilde Z(q, v;\, r)$ and the prefactor
becomes $(q/v^{1/k})^{r(E)}$. Rearranging gives the stated
identity.

Part (3) is immediate by comparing
Definition~\ref{def:abstract_rankpoly} with
Definition~\ref{def:multi_tutte}.
\end{proof}

\subsection{Parity Ising interaction family}
\label{Sec:Isingmodel}

In this section we determine properties of the boolean Parity Ising interaction family $\mcib$.
\begin{proposition}
\label{prop:mcib_properties}
The boolean Parity Ising interaction family $\mcib$ is functionally representable,
contraction--closed, group--coset, and globally satisfiable.
\end{proposition}

\begin{proof}
    The boolean Parity Ising interaction family $\mcib$ satisfies the following properties:
    \begin{enumerate}
        \item \textbf{Group-Coset:} The solution set $K_{\Phi^k} = \{ \vecs \in S^k \mid \sum_{i=1}^k s_i \equiv 0 \mod 2 \}$ is the kernel of the homomorphism $\vecs \mapsto \sum s_i$ over $\mathbb{Z}/2\mathbb{Z}$. Thus, it is a subgroup of $S^k$ (specifically, the even-weight code).
        \item \textbf{Globally Satisfiable:} The all-zeros configuration $\vecs = \veco$ satisfies the parity constraint for all hypergraphs.
        \item \textbf{Functionally Representable:} The linear constraint $\Phi^k(\vecs) = 1 \Leftrightarrow \sum_{i=1}^k s_i = 0$ for $\vecs \in S^k$ allows us to express any chosen variable $s_k$ as the sum of the others: $s_k \equiv \sum_{i=1}^{k-1} s_i \mod 2$. We set $I = \{1, \dots, k-1\}$ and $D=\{k\}$, with the dependency map $\sigma(\vecs_I) = \sum_{j \in I} s_j$.
        \item \textbf{Contraction-closed:} Fix $k \geq 1$, $k' \geq k$, and the decomposition $I = [k-1]$, $D = \{k\}$, $\sigma(\vecs_I) \equiv \sum_{i \in I} s_i \pmod 2$ above. The extended map $\hat\sigma \colon S^{[k'] \setminus D} \to S^{[k']}$ acts by
        \[
            \hat\sigma(\vecs)|_i = \begin{cases} s_i & i \in [k'] \setminus D, \\ \sum_{j \in I} s_j \pmod 2 & i = k. \end{cases}
        \]
        Let $A \subseteq [k']$. For any $\vecs \in S^{[k'] \setminus D}$,
        \begin{equation}
            \label{eq:bparisiidentity}
            \PhibIs^{|A|}(\hat\sigma(\vecs)|_A) = 1 \iff \sum_{i \in A \setminus D} s_i + |A \cap D| \sum_{j \in I} s_j \equiv 0 \pmod 2.
        \end{equation}
        Since $D = \{k\}$, the intersection $A \cap D$ is either empty or equal to $\{k\}$, so $|A \cap D| \in \{0, 1\}$ and the second term is either $0$ or $\sum_{j \in I} s_j$. Set $A' = A \setminus D$ if $A \cap D = \emptyset$ and $A' = (A \setminus D) \,\triangle\, I$ if $A \cap D = \{k\}$, both of which are subsets of $[k'] \setminus D$. We then have $\PhibIs^{|A'|}(\vecs|_{A'}) = \PhibIs^{|A|}(\hat{\sigma}(\vecs)|_A)$ from \eqref{eq:bparisiidentity}. \qedhere
    \end{enumerate}
\end{proof}
By Proposition~\ref{prop:parisboiso}, the interaction families $\mci$
and $\mcib$ are isomorphic with $\alpha_k = 2$ and $\beta_k = -1$ for
every $k$. Propositions~\ref{prop:isoinstant}
and~\ref{prop:equivmodpart} then yield, for any hypergraph $G = (V,E)$,
\begin{equation}
    Z(\Mo(G; \mci); \vecg)
    \;=\; \exp\!\left(- \sum_{e \in E} g_e \right)
          Z(\Mo(G; \mcib); 2 \vecg).
\end{equation}

\begin{corollary}
    \label{cor:isingdelcon}
    Let $G = (V,E)$ be a hypergraph and $e \in E$. Let $\sigma$ be a dependency map of $e$ with respect to $\mcib$. Then the partition function of the hypergraphical model $\Mo(G; \mci)$ satisfies the following deletion contraction relation:
     \[ Z(\Mo(G; \mci); \vecg) = e^{-g_e} \left(Z(\Mo(G \setminus e; \mci); \vecg') + v_e Z(\Mo(G /_{\mcib, \sigma} e; \mci); \vecg') \right),\]
     where $\vecg' = (g_f)_{f \neq e}$.
\end{corollary}

\subsubsection{Incidence matroids}
As a result of Proposition \ref{prop:mcib_properties} the rank function of a hypergraphical boolean Parity Ising model is a polymatroid. In this section we shall show that this polymatroid is, in fact, a binary linear matroid. We start with some preliminary definitions, following \cite{oxleymatroid}.

\begin{definition}
    \label{def:linmat}
    A \emph{linear matroid} $M = (E, r)$ represented by a matrix $B = (\vecmu_1, \dots, \vecmu_k)$ over a vector space $U$ consists of:
    \begin{itemize}
        \item A finite \emph{ground set} $E = \{e_1, \dots, e_k\}$, which indexes the columns of $B$.
        \item A \emph{rank function} $r: 2^E \to \mathbb{Z}_{\geq0}$ that assigns to each subset $A \subseteq E$ the dimension of the vector space spanned by the corresponding column vectors of $A$, i.e.,
        \[
            r(A) = \dim \operatorname{im} B|_A,
        \]
        where $B|_A$ corresponds to the submatrix of $B$ with columns indexed by $A$.
    \end{itemize}
\end{definition}

A linear matroid is called \textit{binary} if it can be represented by a matrix over $\mathbb{F}_2$.

We can associate a matroid to a hypergraph using its incidence matrix over $\mathbb{F}_2$ as follows.
 
\begin{definition}
    Let $G = (V,E = \{e_1, \dots, e_k\})$ be a hypergraph with $|V| = n$ and $|E| = k$ and let $B(G)$ be its incidence matrix over $\mathbb{F}_2$, which is defined by
    \[ B(G) = (\vecmu_{1}, \dots, \vecmu_{k}),\]
    where $\vecmu_{i}$ is the $n$-dimensional indicator column vector with a $1$ at index $j$ if and only if $j \in e_i$. The binary \emph{incidence} matroid $M(G) = (E, r)$ corresponding to the hypergraph $G$ is now defined as the matroid represented by the incidence matrix $B(G)$.
\end{definition}

\begin{proposition}
\label{prop:parity_rank_equals_matroid}
    Let $G = (V,E = \{e_1, \dots, e_k \})$ be a hypergraph with $|V| = n$ and let $A \subseteq E$, then:
    \[ r_{\Mo(G; \mcib)}(A) = r_{M(G)}(A).\]
\end{proposition}
\begin{proof}
    To see this, we start by noting that 
    \begin{align*}
    r_{\Mo(G; \mcib)}(A) &= n- \log_2 (\# \cap_{e \in A} K_{\phi^e}) \\
    &= n - \log_2 \left(\# \{ \vecs \in S^n \, \colon \, \sum_{i \in e} s_i \equiv 0 \mod 2 \quad \forall e \in A \}\right).
    \end{align*}
    Note that, if we view $\vecs$ as a column vector over $\mathbb{F}_2$, then $\sum_{i \in e_j} s_i = \vecs \cdot \vecmu_j$. So that 
    \[ \sum_{i \in e} s_i \equiv 0 \mod 2 \quad \forall e \in A  \Leftrightarrow B(G)|_A^T \cdot \vecs = 0 \Leftrightarrow \vecs \in \ker B(G)|_A^T, \]
    where $B(G)|_A^T$ denotes the transpose of the incidence matrix so that its rows correspond to the edges in $A$. Now $\log_2 (\# \ker B(G)|_A^T) = \dim \ker B(G)|_A^T$, where we take the dimension over $\mathbb{F}_2$. So we conclude that 
    \[
    r_{\Mo(G; \mcib)}(A) = n-\dim \ker B(G)|_A^T.
    \]
    It is now a standard result from linear algebra that
    \[ n - \dim \ker B(G)|_A^T  = \dim \operatorname{im} B(G)|_A^T, \]
    and since row rank equals column rank this proves the statement.
\end{proof}

\begin{remark}[Gauge invariance and matroid representatives]
\label{rem:gauge}
Proposition~\ref{prop:parity_rank_equals_matroid} shows that
$r_{\Mo(G; \mcib)}$ depends on $G$ only through the binary matroid $M(G)$ represented by the incidence matrix $B(G)$. Combined with Corollary~\ref{cor:rank_determines_partition}, this means that any two hypergraphs on a common vertex set whose incidence matrices represent the same matroid have identical boolean Parity Ising partition functions. Concretely, left-multiplying $B(G)$ by an invertible matrix $P \in \mathrm{GL}_{|V|}(\mathbb{F}_2)$ is a change of basis of $\mathbb{F}_2^{V}$ that preserves all linear dependencies among the columns, hence the matroid $M(G)$, and therefore leaves the partition function unchanged. Transformations of this kind are precisely the \emph{gauge transformations} of~\cite{spinmod}, where they are shown by a direct computation to preserve the partition function; here their invariance is an immediate consequence of the matroid-theoretic
description of the rank function. 
\end{remark}

\begin{remark}[Connection to the Tutte polynomial]
\label{rem:parity_tutte}
By Proposition~\ref{prop:parity_rank_equals_matroid}, $r_{\Mo(G; \mcib)}$ is the rank function of the binary matroid $M(G)$. Proposition~\ref{prop:rankpoly_specializations}(3) therefore identifies the rank generating function of $\Mo(G; \mcib)$ with the multivariate Tutte polynomial of $M(G)$:
\[
    \widetilde Z(2, \vecv;\, r_{\Mo(G; \mcib)})
    \;=\;
    \widetilde Z_{\mathrm{Tutte}}(2, \vecv;\, M(G)).
\]
For graphs $G$, $M(G)$ coincides with the cycle matroid, so the counting evaluations of Theorem~\ref{theo:counteval} recover known evaluations of the Tutte polynomial; see e.g.~\cite{ellis-monaghanHandbookTuttePolynomial2022}. In particular, setting $v_e = -1$ for every $e \in E$ gives the classical chromatic polynomial of $G$ evaluated at $q = 2$:
\[
    \widetilde Z\bigl(2, (-1)_{e \in E};\, r_{\Mo(G; \mcib)}\bigr)
    \;=\; \frac{P(G; 2)}{2^{|V|}},
\]
where $P(G; q) = \widetilde Z_{\mathrm{Tutte}}(q, (-1)_{e\in E};\, M(G))$ denotes the chromatic polynomial of $G$. We can see this directly from Theorem~\ref{theo:counteval}: for an edge $e = \{i,j\}$ and $\vecs \in S^V$, the constraint $\phi^e(\vecs) = 1$ becomes $s_i \neq s_j$, so a configuration with $\phi^e(\vecs) = 0$ for every $e \in E$ is precisely a proper $2$-coloring of $G$. Hence
\[
    \#\{\vecs \in S^V \mid \phi^e(\vecs) = 0 \text{ for all } e \in E\}
    \;=\; P(G; 2),
\]
and dividing by $2^{|V|}$ recovers the rank generating function via Proposition~\ref{prop:boolmodcountex}.
\end{remark}

\subsubsection{Deletion-contraction for Parity Ising models}
Here we explicate the construction of the contracted hypergraph. Recall that the interaction for the boolean Parity Ising model $\mcib$ on an edge $e$ imposes the constraint $\sum_{i \in e} s_i \equiv 0 \mod 2$. This linear constraint allows us to express any variable $s_v$ (for $v \in e$) as the sum of the remaining variables. Thus, the dependency map $\sigma$ is determined by the choice of a \emph{pivot vertex} $v \in e$.

\begin{proposition}[Induced Contraction for Boolean Parity Ising Models]
    \label{def:parity_contraction}
    Let $G = (V, E)$ be a hypergraph and let $e \in E$. Let $e = I \sqcup D$ where $D = \{v\}$ is the chosen pivot vertex and let $\sigma$ be the corresponding dependency map. Then $G/_{ \mcib, \sigma}e$ is the hypergraph $(V', E')$ defined as follows:
    \begin{itemize}
        \item The vertex set is $V' = V \setminus \{v\}$.
        \item The hyperedge multiset is $E' = \{ f \Delta_v e \mid f \in E \setminus \{e\} \}$,
    \end{itemize}
    where the \emph{pivot-based symmetric difference} is defined by:
    \[
    f \Delta_v e = \begin{cases}
        (f \Delta e) \setminus \{v\} & \text{if } v \in f \\
        f & \text{if } v \notin f
    \end{cases}.
    \]
\end{proposition}
\begin{proof}
    This is a direct result of construction of $A'$ in the proof of contraction closure in Proposition \ref{prop:mcib_properties}.
\end{proof}

\begin{remark}[Pivot Dependence]
    As noted in the general case, the structure of the contracted hypergraph $G/_{ \mcib, \sigma}e$ depends on the choice of dependency map $\sigma$ and thus on the choice of pivot vertex. For general hypergraphs, choosing different pivots yields non-isomorphic hypergraphs, while the partition functions of their corresponding hypergraphical models are the same.
\end{remark}

\subsubsection{Graphs with blisters and the Ising model in the presence of an external field}
\label{Section:blisters}
The ambiguity of the contraction operation vanishes when we restrict our attention to a specific subclass of hypergraphs that models the classical Ising model with an external field. In this case we shall show that, for a loopless graph $G = (V,E)$ and $e \in E$, the contracted graph $G/_{\mcib, \sigma}e$ corresponds to regular edge contraction known from graph theory and does not depend on the chosen pivot vertex. We shall extend this pivot independence to graphs with blisters and show that graph with blisters model Ising models in the presence of an external field.

\begin{definition}
    A \emph{loopless graph with blisters} is a loopless hypergraph $G=(V,E)$ where every hyperedge has cardinality $|e| \in \{1, 2\}$. Edges with $|e|=1$ are called \emph{blisters}.
\end{definition}

In the context of the boolean Parity Ising model, a blister $\{i\}$ yields the term $\exp(g_i s_i)$ in the maximum-entropy distribution of the model, which functions as an external field. To see this, let $G = (V,E)$ be a loopless graph with blisters. Furthermore, we decompose $E = V' \sqcup E'$ where $V' \subseteq E$ is the subset of blisters and $E'$ is the subset of edges, i.e. $|e| = 2$ for $e \in E'$. The maximum-entropy distribution of $\Mo(G; \mcib)$ is given by
\[
p(\vecs \mid \Mo(G; \mcib), \vecg) = \frac{1}{Z(\Mo; \vecg)}\left[ \prod_{\{i\} \in V'} \exp(g_{\{i\}} s_i) \right] \left[ \prod_{\{i,j\} \in E'} \exp(g_{\{i,j\}} s_{i}s_j) \right].
\]

\begin{proposition}[Uniqueness of Contraction for Graphs with Blisters]
    \label{prop:blister_unique}
    Let $G$ be a loopless graph with blisters and let $e = \{u, v\}$ be an edge of size 2. Let $\sigma^v$ and $\sigma^u$ be the dependency maps corresponding to the vertices $u$ and $v$, then both dependency maps yield isomorphic contracted graphs with blisters:
    \[ G /_{\mcib, \sigma^u} e \cong G /_{\mcib, \sigma^v} e. \]
    Specifically, the operation corresponds to the standard vertex identification of $u$ and $v$, where any blister on the removed vertex is transferred to the identified vertex.
\end{proposition}

\begin{proof}
    Consider the contraction at pivot $u$. The vertex set becomes $V \setminus \{u\}$.
    \begin{itemize}
        \item Any edge $f=\{u, z\}$ with $z\neq v$ incident to $u$ becomes:
        \[ f' = \{u, z\} \Delta \{u, v\} \setminus \{u\} = \{v, z\}. \]
        This replaces the connection to $u$ with a connection to $v$. If $f = \{u,v\}$ is an edge parallel to $e$, then $f$ becomes $f'=\emptyset$.
        \item Any blister $b=\{u\}$ at $u$ becomes:
        \[ b' = \{u\} \Delta \{u, v\} \setminus \{u\} = \{v\}. \]
        This transfers the blister from $u$ to $v$.
    \end{itemize}
    The symmetric argument holds for pivot $v$. Thus, contraction in graphs with blisters is canonically defined as vertex identification, summing parallel edges and blisters (modulo 2) where they arise.
\end{proof}

This recovers the standard behavior of the Ising model in an external field: contracting an edge identifies the vertices and adds their external fields.

\begin{remark}
   In \cite{zbMATH05956368}, a deletion contraction result for Ising models on graphs in the presence of an external field is obtained. In our framework this corresponds to a graph with blisters. If we consider a graph where all blisters are present and contract an edge $e = \{u,v\}$ the resulting graph will have two parallel blisters on the new vertex. Using Proposition \ref{prop:edge_redundancy} the two parallel blisters can be reduced to one. Following this procedure gives the contraction results obtained in \cite{zbMATH05956368} in a matroid-theoretic framework.
\end{remark}

\subsection{Delta Potts interaction family}
\label{sec:potts}

In this section, we apply our general framework to the Delta Potts interaction family $\mcp$ introduced previously. We explicitly verify its algebraic properties, identify the associated polymatroid, and derive the combinatorial contraction operation.

\begin{proposition}[Properties of $\mcp$]
    The Delta Potts interaction family $\mcp = (S, \{\Phi^k\}_{k \in \mathbb{Z}_{\geq 0}})$ is group-coset, globally satisfiable, functionally representable and contraction-closed.
\end{proposition}
\begin{proof}
    The Delta Potts interaction family $\mcp$ satisfies the following properties:
    \begin{enumerate}
        \item \textbf{Group-Coset:} The solution set $K_{\Phi^k} = \{ (s, \dots, s) \mid s \in S \}$ corresponds to the diagonal subgroup of the abelian group $S^k$ (where $S$ is viewed as $\mathbb{Z}_q$). Thus, it is a group-coset interaction family.
        \item \textbf{Globally Satisfiable:} The constant configuration $\vecs = \veco$ satisfies $\Phi^k(\veco) = 1$ for all $k$. Thus, the interaction family is globally satisfiable.
        \item \textbf{Functionally Representable:} For $k \in \mathbb{Z}_{>0}$ the constraint $s_i = s_j$ for all $i,j \in [k]$ allows us to express all variables in $[k-1]$ as equal to $s_k$. We set the independent set $I=\{k\}$ and dependent set $D = [k-1]$, with dependency map $\sigma(s_k) = (s_k, \dots, s_k)$.
        \item \textbf{Contraction-closed:} Fix $k \geq 1$, $k' \geq k$, and the
decomposition $I = \{k\}$, $D = [k-1]$, $\sigma(s_k) = (s_k, \dots, s_k)$
above. The extended map $\hat\sigma \colon S^{[k'] \setminus D} \to S^{[k']}$
acts by
\[
    \hat\sigma(\vecs)|_i =
    \begin{cases}
        s_k & i \in [k-1], \\
        s_i & i \in \{k, k+1, \dots, k'\}.
    \end{cases}
\]
In particular, $\hat\sigma(\vecs)|_{[k]}$ is constantly equal to $s_k$. Let
$A \subseteq [k']$; then $\Phi^{|A|}(\hat\sigma(\vecs)|_A) = 1$ if and
only if all coordinates of $\hat\sigma(\vecs)|_A$ are equal. We
distinguish three cases.

\emph{Case 1: $A \cap [k] = \emptyset$.} Then $A \subseteq
\{k+1,\dots,k'\} \subseteq [k'] \setminus D$ and $\hat\sigma(\vecs)|_A =
\vecs|_A$, so $\Phi^{|A|}(\hat\sigma(\vecs)|_A) = \Phi^{|A|}(\vecs|_A)$
and we may take $A' = A$.

\emph{Case 2: $A \subseteq [k]$.} Then $\hat\sigma(\vecs)|_A$ is constantly
$s_k$, so $\Phi^{|A|}(\hat\sigma(\vecs)|_A) = 1$ for all $\vecs$, i.e.\ the
function is constant.

\emph{Case 3: $A \cap [k] \neq \emptyset$ and $A \not\subseteq [k]$.} The
coordinates of $\hat\sigma(\vecs)|_A$ indexed by $A \cap [k]$ are all
equal to $s_k$, while those indexed by $A \cap \{k+1,\dots,k'\}$ equal
$s_i$. All being equal is equivalent to $s_k = s_j$ for every $j \in
A \cap \{k+1,\dots,k'\}$, which is exactly
$\Phi^{|A'|}(\vecs|_{A'})$ for $A' := \{k\} \cup (A \cap
\{k+1,\dots,k'\}) \subseteq [k'] \setminus D$.

In each case the function is either constant or equals
$\Phi^{|A'|}(\vecs|_{A'})$ for some unique $A' \subseteq [k'] \setminus D$,
verifying contraction-closedness. \qedhere
\end{enumerate}
\end{proof}

\subsubsection{Polymatroids and partition functions}

Since $\mcp$ is a globally satisfiable group-coset interaction family satisfying functional representability and contraction--closure, Corollary~\ref{cor:rank_is_polymatroid} ensures that the rank function of any model $\Mo(G; \mcp)$ is a polymatroid.

The rank function of the Delta Potts model, like that of the boolean Parity Ising model, corresponds to a known polymatroid associated to hypergraphs. The hypergraphical polymatroid was introduced in~\cite{Helg} under the name \emph{chromatic hypermatroid}.
\begin{definition}
    Let $G=(V,E)$ be a loopless hypergraph with $n=|V|$. The \emph{hypergraphical polymatroid} is defined via the following rank function:
    \[
    \chi_G(A) = n - \kappa(G_A) \quad \text{for } A \subseteq E,
    \]
    where $\kappa(G_A)$ is the number of connected components in the spanning subhypergraph $(V, A)$.
\end{definition}
See e.g. \cite{Vertigan_Whittle_1993} for further information about hypergraphical polymatroids.

\begin{proposition}[Delta Potts Rank Function]
\label{prop:delta_potts_rank}
    Let $G=(V,E)$ be a loopless hypergraph with $n=|V|$. The rank function $r_{\Mo(G; \mcp)}$ associated with the Delta Potts model is the hypergraphical polymatroid
    \[ r_{\Mo}(A) = n - \kappa(G_A) = \chi_G(A)  \quad \text{for } A \subseteq E. \]
\end{proposition}

\begin{proof}
    The counting function $R_{\Mo}(A)$ counts the number of configurations such that $s_i = s_j$ for all edges $e \in A$. This condition forces all spins within any connected component of the hypergraph $(V, A)$ to be identical. Since there are $\kappa(G_A)$ such components and $q$ choices for the spin value of each component, we have $R_{\Mo}(A) = q^{\kappa(G_A)}$.
    The rank function is therefore:
    \[ r_{\Mo}(A) = n - \log_q(R_{\Mo}(A)) = n - \kappa(G_A). \]
\end{proof}

\begin{remark}[Connection to the Poincar\'e polynomial]
\label{rem:delta_potts_poincare}
By Proposition~\ref{prop:rankpoly_specializations}(1) applied to $r_{\Mo(G; \mcp)} = \chi_G$, the rank generating function with all edge variables specialized to $\eta - 1$ recovers the Poincar\'e polynomial of the hypergraphical polymatroid:
\[
    \widetilde Z\bigl(q,\, (\eta - 1)_{e \in E};\, \chi_G\bigr)
    \;=\;
    q^{-\chi_G(E)} \, \tau(q, \eta;\, \chi_G).
\]
Combining with Proposition~\ref{prop:boolmodcountex} and the identity $\chi_G(E) = n - \kappa(G)$ gives
\[
    Z(\Mo(G; \mcp); \vecg)\big|_{v_e = \eta - 1}
    \;=\;
    q^{\kappa(G)} \, \tau(q, \eta;\, \chi_G).
\]
Evaluating $\tau(q, \eta;\, \chi_G)$ at $\eta = 0$, or equivalently $v_e =-1$, recovers the chromatic-polynomial identity from \cite[Theorem~3.14]{Helg}: the number of weak $q$-colorings of $G$ equals $q^{\kappa(G)} \cdot \tau(q, 0;\, \chi_G)$. We can see this directly from Theorem~\ref{theo:counteval}: a configuration $\vecs \in [q]^V$ with $\phi^e(\vecs) = 0$ for every $e \in E$ is precisely one in which no edge is monochromatic, i.e.\ a weak $q$-coloring of $G$. In~\cite{Helg} the counting result is stated for connected hypergraphs, where $q^{\kappa(G)} = q$ and the formula reduces to $q \cdot \tau(q,0;\, \chi_G)$.
\end{remark}

\subsubsection{Deletion-contraction}

We now define the specific contraction operation on hypergraphs induced by the functional representation of the Delta Potts model.

\begin{proposition}[Induced Hypergraph Contraction for Delta Potts]
    \label{prop:potts_contraction}
    Let $G=(V,E)$ be a hypergraph and let $e \in E$. Let $e = I \sqcup D$ be a decomposition with respect to $\mcp$ where $I = \{v\}$ is a chosen pivot vertex and $D = e \setminus \{v\}$ and let $\sigma$ be the corresponding dependency map. Then $G/_{\mcp, \sigma}e$ is the hypergraph $(V', E')$ defined as follows:
    \begin{itemize}
        \item The vertex set is $V' = V \setminus D$.
        \item The hyperedge multiset is $E' = \{ \pi(f) \mid f \in E \setminus \{e\} \}$,
    \end{itemize}
    where $\pi$ is the projection map that identifies all vertices in $D$ with $v$:
    \[
    \pi(u) = \begin{cases}
        v & \text{if } u \in D, \\
        u & \text{otherwise}.
    \end{cases}
    \]
\end{proposition}

\begin{remark}[Relation to the hypergraph Tutte polynomial]
\label{rem:potts_contraction_bem}
In~\cite{Berrekkal2026} a hypergraph Tutte polynomial is introduced,
defined for a hypergraph $H = (V,E)$ by
\[
    T_{\mathrm{HG}}(H; x, y)
    \;=\;
    \sum_{A \subseteq E} (x-1)^{\kappa(G_A) - \kappa(H)}\,
    (y-1)^{d(A) - |A| - |V| + \kappa(G_A)},
\]
where $d(A) = \sum_{e \in A} |e|$ is the total edge degree. This relates
to our rank generating function for the Delta Potts model on
hypergraphs through a degree-dependent specialization of the edge
variables, which absorbs the hyperedge degrees into the variables rather
than requiring uniformity. Since the Delta Potts rank function is the
hypergraphical polymatroid, $r_\Mo(A) = \chi_G(A) = |V| - \kappa(G_A)$
(Proposition~\ref{prop:delta_potts_rank}), setting $q = (x-1)(y-1)$ and
$v_e = (y-1)^{|e|-1}$ for each $e \in E$ gives
\[
    \prod_{e \in A} (y-1)^{|e|-1} = (y-1)^{d(A) - |A|},
    \qquad
    q^{-\chi_G(A)} = (x-1)^{\kappa(G_A) - |V|} (y-1)^{\kappa(G_A) - |V|},
\]
so that
\[
    T_{\mathrm{HG}}(H; x, y)
    \;=\;
    (x-1)^{|V| - \kappa(H)}\,
    \widetilde Z\!\bigl( (x-1)(y-1),\, ((y-1)^{|e|-1})_{e \in E};\, \chi_G \bigr).
\]

In~\cite{Berrekkal2026} they also establish a deletion--contraction recurrence for the hypergraph Tutte polynomial. Their hypergraph contraction operation is defined on general hypergraphs, i.e. not loopless hypergraphs. By absorbing the edge multiplicities into the parameters one can reduce their contraction to a contraction operation on loopless hypergraphs and one finds that their deletion--contraction recurrence, in fact, agrees with ours.   
\end{remark}

\subsection{And Ising interaction family}
\label{sec:AndIsing}

In this section, we examine the And Ising interaction family, which exhibits a distinct combinatorial structure related to the boolean polymatroid.

\begin{proposition}[Properties of $\mca$]
    The And Ising interaction family $\mca$ is group-coset, globally satisfiable, functionally representable and contraction-closed.
\end{proposition}
\begin{proof}
    The And Ising interaction family satisfies the following properties:
    \begin{enumerate}
        \item \textbf{Group-Coset:} The solution set $K_{\Phi^k} = \{ (1, \dots, 1) \}$ is a singleton. Viewed in the group $\mathbb{Z}_2^k$, this is the coset $\vecone + \{\veco\}$, where $\{\veco\}$ is the trivial subgroup.
        \item \textbf{Globally Satisfiable:} The configuration $\vecs = \vecone$ satisfies all interactions.
        \item \textbf{Functionally Representable:} Let $k \in \mathbb{Z}_{>0}$. The constraint $\Phi^k(\vecs) = \prod_{i=1}^k s_i = 1$ forces $s_i = 1$ for all $i$. We choose the independent set $I = \emptyset$ and dependent set $D = [k]$, with the constant dependency map $\sigma(\emptyset) = \vecone$.
        \item \textbf{Contraction-closed:} Fix $k \geq 0$, $k' \geq k$, and the
decomposition $I = \emptyset$, $D = [k]$, $\sigma(\emptyset) = \vecone \in
S^k$ above. The extended map $\hat\sigma \colon S^{[k'] \setminus D} \to
S^{[k']}$ acts by
\[
    \hat\sigma(\vecs)|_i =
    \begin{cases}
        1 & i \in [k], \\
        s_i & i \in \{k+1, \dots, k'\}.
    \end{cases}
\]
For any $A \subseteq [k']$,
\[
    \Phi^{|A|}_{\mathrm{And-Is}}(\hat\sigma(\vecs)|_A)
    = \prod_{i \in A} \hat\sigma(\vecs)|_i
    = \prod_{i \in A \cap [k]} 1 \;\cdot\; \prod_{i \in A \setminus [k]} s_i
    = \prod_{i \in A \setminus [k]} s_i.
\]
If $A \subseteq [k]$, the product is empty and the function is constant. Otherwise we set $A' := A \setminus [k] \subseteq [k'] \setminus D$; which equals $\Phi^{|A'|}_{\mathrm{And-Is}}(\vecs|_{A'})$. \qedhere
\end{enumerate}
\end{proof}

\subsubsection{Polymatroids and partition functions}
Since $\mca$ satisfies the hypotheses of Corollary~\ref{cor:rank_is_polymatroid}, the rank function of any model $\Mo(G; \mca)$ is a polymatroid. Once again, this polymatroid is a known polymatroid associated to hypergraphs. The boolean polymatroid was introduced in~\cite{Helg} under the name \emph{covering hypermatroid}.

\begin{definition}
     Let $G=(V,E)$ be a hypergraph. The \emph{boolean polymatroid} associated to $G$ is defined as
     \[
     \rho_G(A) = |\overline{A}| = \left| \bigcup_{e \in A} e \right|.
     \]
\end{definition}

\begin{proposition}[Boolean Rank Function]
\label{prop:and_ising_rank}
    Let $\Mo(G; \mca)$ be a hypergraphical model for some hypergraph $G = (V,E)$ with $|V| = n$. The rank function $r_{\Mo(G; \mca)}$ equals the rank function of the \emph{boolean polymatroid}:
    \[ r_{\Mo}(A) = |\overline{A}| = \left| \bigcup_{e \in A} e \right| = \rho_G(A). \]
\end{proposition}

\begin{proof}
    The counting function $R_{\Mo}(A)$ counts configurations where $s_i=1$ for all $i \in e$ for every $e \in A$. This requires $s_v = 1$ for all $v \in \bigcup_{e \in A} e = \overline{A}$. The spins for vertices not in $\overline{A}$ are unconstrained. 
    Thus, there are $2^{n - |\overline{A}|}$ satisfying configurations.
    The rank function is:
    \[ r_{\Mo}(A) = n - \log_2(2^{n - |\overline{A}|}) = |\overline{A}| = \rho_G(A). \]
\end{proof}

\begin{remark}[Connection to the Poincar\'e polynomial]
\label{rem:and_ising_poincare}
By Proposition~\ref{prop:rankpoly_specializations}(1) applied to $r_{\Mo(G; \mca)} = \rho_G$, the rank generating function with all edge variables specialized to $\eta - 1$ recovers the Poincar\'e polynomial of the boolean polymatroid:
\[
    \widetilde Z\bigl(2,\, (\eta - 1)_{e \in E};\, \rho_G\bigr)
    \;=\;
    2^{-\rho_G(E)} \, \tau(2, \eta;\, \rho_G).
\]
Combining with Proposition~\ref{prop:boolmodcountex} gives
\[
    Z\bigl(\Mo(G; \mca); \vecg\bigr)\Big|_{v_e = \eta - 1}
    \;=\;
    2^{|V| - \rho_G(E)} \, \tau(2, \eta;\, \rho_G).
\] 
Setting $\eta = 0$, or equivalently $v_e = -1$, recovers the evaluation
result from~\cite[Corollary~3.10]{Helg}: $2^{|V| - \rho_G(E)}\tau(2,0;\, \rho_G)$ counts
the transversal sets of $G$, i.e.\ the subsets $T \subseteq V$ with
$T \cap e \neq \emptyset$ for every $e \in E$. The correspondence is
direct: for $\vecs \in \{0,1\}^V$, the And interaction
$\phi^e(\vecs) = \prod_{i \in e} s_i$ vanishes iff $e$ contains some
$i$ with $s_i = 0$, so $\phi^e(\vecs) = 0$ for every $e \in E$ iff
the zero-set $\{i \in V \mid s_i = 0\}$ is a transversal of $G$. By
Theorem~\ref{theo:counteval}, this count is the $v_e \equiv -1$
evaluation of the partition function. In~\cite{Helg} the counting result is stated for hypergraphs without isolated vertices, where $\rho_G(E) = |V|$ and the formula reduces to $\tau(2,0;\, \rho_G)$.
\end{remark}

\subsubsection{Deletion-contraction}

\begin{proposition}
    \label{prop:and_contraction}
    Let $G=(V,E)$ be a hypergraph and $e \in E$. Let $\sigma$ be a dependency map of $e$ with respect to $\mca$. Then $G/_{\mca, \sigma}e$ is the hypergraph defined by removing the vertices of $e$.
    
    Formally, the contracted hypergraph has vertex set $V' = V \setminus e$ and edge multiset:
    \[ E' = \{ f \setminus e \mid f \in E \setminus \{e\} \}. \]
\end{proposition}

\section{Discussion}
\label{sec:discussion}
In this paper we developed a rigorous framework for statistical mechanics models on hypergraphs. We defined \emph{interaction families} that systematically induce a hypergraphical model on every hypergraph. Restricting to boolean loopblind interaction families, we identified two structural conditions, functional representability and contraction-closedness, that together yield a Tutte-like deletion--contraction recurrence for the partition function (Theorem~\ref{theo:contraction_closure}, Corollary~\ref{cor:partition_delcon}). We further identified two conditions, group--coset structure and global satisfiability, that together with the structural conditions force the rank function of any induced model to be a polymatroid (Corollary~\ref{cor:rank_is_polymatroid}).

We illustrated the theory on three boolean interaction families generalizing classical graph models to hypergraphs. All three satisfy all four conditions, and the resulting polymatroids are three distinct classical combinatorial structures associated to a hypergraph. Table~\ref{tab:models} summarizes these examples.

\begin{table}[ht]
\centering
\caption{The three interaction families studied in this paper. In each case the partition function is related to a multivariate polynomial associated to the (poly)matroid: the multivariate Tutte polynomial of the incidence matroid for boolean Parity Ising (Remark~\ref{rem:parity_tutte}), and a multivariate version of the Poincar\'e polynomial of the hypergraphical and boolean polymatroids for Delta Potts and And Ising respectively (Remarks~\ref{rem:delta_potts_poincare} and~\ref{rem:and_ising_poincare}). Under a degree-dependent specialization of the edge variables, the Delta
Potts rank generating function also recovers the hypergraph Tutte
polynomial of~\cite{Berrekkal2026}
(Remark~\ref{rem:potts_contraction_bem}).}
\label{tab:models}
\renewcommand{\arraystretch}{1.4}
\begin{tabular}{c l c c}
\hline\hline
 & Interaction $\Phi^k(\vecs)$
 & Name
 & Polymatroid \\
\hline
\ref{Sec:Isingmodel}
 & $\delta(\sum_{i=1}^k s_i \bmod 2,0)$
 & boolean Parity Ising
 & incidence matroid \\[4pt]
\ref{sec:potts}
 & $\delta_k(\vecs)$
 & Delta Potts
 & hypergraphical \\[4pt]
\ref{sec:AndIsing}
 & $\displaystyle\prod_{i=1}^{k} s_i$
 & And Ising
 & boolean \\
\hline\hline
\end{tabular}
\end{table}

On graphs, the boolean Parity Ising and Delta Potts ($q=2$) interaction families are isomorphic as hypergraphical models and have identical rank functions, but on hypergraphs they correspond to different polymatroids, namely the incidence matroid versus the hypergraphical polymatroid. This is consistent with our non-isomorphism result (Proposition~\ref{prop:potts_parity_noniso}) and shows that the multivariate Tutte polynomial of a graph admits at least two genuinely distinct hypergraph generalizations within our framework, depending on which classical graph model one chooses to lift.

The framework also offers a unifying view on classical hypergraph polynomial invariants. For the Delta Potts and And Ising families, the rank generating function is a multivariate version of the Poincar\'e polynomial of the associated polymatroid (Remarks~\ref{rem:delta_potts_poincare} and~\ref{rem:and_ising_poincare}); for the boolean Parity Ising family, the rank function is a matroid, and the rank generating function is the multivariate Tutte polynomial of the binary incidence matroid (Remark~\ref{rem:parity_tutte}). Specializations of the rank generating function at $v_e = -1$ for the Delta Potts and And Ising interaction families recover hypergraph counting identities for weak colorings and transversal sets.

Several further results follow from our framework. First, the partition function of a boolean hypergraphical model only depends on its rank function (Corollary~\ref{cor:rank_determines_partition}). For the Parity Ising family the rank function is the incidence matroid (Proposition~\ref{prop:parity_rank_equals_matroid}) and all hypergraphs whose incidence matrices represent the same binary matroid form a class with a common partition function, recovering and explaining the gauge invariance of~\cite{spinmod} (Remark~\ref{rem:gauge}). Second, for the boolean Parity Ising family, external fields are modeled in our framework by blisters and our deletion--contraction on graphs with blisters recovers, in matroid-theoretic form, the results of~\cite{zbMATH05956368} about the Ising model in the presence of an external field
(Section~\ref{Section:blisters}).

\subsection*{Future directions}

Several natural directions arise from this work. Functional representability and con\-traction-closedness are sufficient for the deletion--contraction recurrence, and group--coset structure together with global satisfiability are sufficient for the rank function to be a polymatroid. We have not investigated whether any of these conditions are necessary, and a natural next step is to determine which are required for a well-defined contraction relation and a polymatroid rank function, and which can be relaxed. Relatedly, the three boolean families studied here all satisfy every property we introduced, and a broader question is how rich the landscape of interaction families satisfying these conditions actually is. Finally, our results rely on interaction functions taking values in $\{0,1\}$. 
The vector Potts family (Example~\ref{ex:vectorpotts}) is the most immediate example outside of this setting, for which the interaction functions can take more than two values (the XY~and Blume-Emery-Griffiths models are other examples).
Investigating whether an analogous deletion--contraction and polymatroid framework exists for non-boolean families is a natural direction for future work.

\begin{ack}
CdM is grateful to Italo Simonelli for insightful discussions.
\end{ack}

\begin{funding}
Merijn Moody is funded by the Dutch Institute for Emergent Phenomena (DIEP)
at the University of Amsterdam, via the DIEP programme Foundations and Applications of Emergence (FAEME). Khallil Berrekkal is supported by the Netherlands Organisation for Scientific Research (NWO) under grant VI.Vidi.193.068.
\end{funding}

\bibliographystyle{emss} 
\bibliography{ref} 
\end{document}